\documentclass[12pt]{amsart}

 \usepackage{amsmath}
 \usepackage{amssymb}

\usepackage{xcolor}

\DeclareMathOperator{\vol}{vol}
\newcommand{\omegaij}{\omega_{\widehat{ij}}}

\title[Volume and cusp excursions in hyperbolic geodesics]{The volume of a divisor and cusp excursions of geodesics in hyperbolic manifolds}

\newcommand{\R}{\mathbb R}

\renewcommand{\P}{\mathbb P}
\newcommand{\floor}[1]{\left\lfloor #1 \right\rfloor}

\newcommand{\ti}[1]{\tilde{#1}}

\newif\ifdebug
 \debugfalse

\newif \iffig
\figtrue
\newif \iftable
\tablefalse

\usepackage{sfilip_package_settings}
\usepackage{sfilip_thm_style_long}
\usepackage{sfilip_abbreviations}

\PassOptionsToPackage{hyphens}{url}\usepackage{hyperref}

\makeatletter
\g@addto@macro{\UrlBreaks}{\UrlOrds}
\makeatother
\usepackage{color}
\definecolor{darkred}{rgb}{0.4,0,0}
\definecolor{darkgreen}{rgb}{0,0.5,0}
\definecolor{darkblue}{rgb}{0,0,0.4}
\hypersetup{
	pdftitle={},	% title
	pdfauthor={Simion Filip, John Lesieutre, Valentino Tosatti},	% author
	pdfsubject={math},		% subject of the document
	pdfkeywords={MSC},		% list of keywords
	pdfnewwindow=true,		% links in new window
	colorlinks=true,		% false: boxed links; true: colored links
	linkcolor=darkblue,		% color of internal links (change box color with linkbordercolor)
	citecolor=darkred,		% color of links to bibliography
	filecolor=darkblue,		% color of file links
	urlcolor=darkblue,		% color of external links
	pdfborder={0 0 0},
	breaklinks=true
}

\thanks{Revised \textsc{\today}}
\date{December 2023}

\author{
	Simion Filip
}
\address{
	\parbox{0.7\textwidth}{
		Department of Mathematics\\
		University of Chicago\\
		5734 S University Ave\\
		Chicago, IL 60637}
}
\email{{sfilip@math.uchicago.edu}}

\author{
	John Lesieutre
}
\address{
	\parbox{0.7\textwidth}{
		Department of Mathematics\\
		The Pennsylvania State University\\
		204 McAllister Building\\
		University Park, PA 16802}
}
\email{{jdl@psu.edu}}

\author{
	Valentino Tosatti
}
\address{
	\parbox{0.7\textwidth}{Courant Institute of Mathematical Sciences\\
New York University\\
251 Mercer St\\
New York, NY 10012}
}
\email{{tosatti@cims.nyu.edu}}
\begin{document}

\begin{abstract}
We give a complete description of the behavior of the volume function at the boundary of the pseudoeffective cone of certain Calabi--Yau complete intersections known as Wehler $N$-folds.
We find that the volume function exhibits a pathological behavior when $N\geq 3$, we obtain examples of a pseudoeffective $\mathbb{R}$-divisor $D$ for which the volume of $D+sA$, with $s$ small and $A$ ample, oscillates between two powers of $s$, and we deduce the sharp regularity of this function answering a question of Lazarsfeld.
We also show that $h^0(X,\floor{mD}+A)$ displays a similar oscillatory behavior as $m$ increases, showing that several notions of numerical dimensions of $D$ do not agree and disproving a conjecture of Fujino.
We accomplish this by relating the behavior of the volume function along a segment to the visits of a corresponding hyperbolic geodesics to the cusps of a hyperbolic manifold.
\end{abstract}

\maketitle

\listoffixmes

\section{Introduction}

\subsection{The volume function}
\subsubsection{Volume of divisors} Let $X^n$ be a smooth projective variety over $\mathbb{C}$, and $D$ a Cartier divisor on $X$. As shown by Lazarsfeld \cite[\S 2.2.C]{pag1}, its volume
\[\vol(D)=\limsup_{m\to+\infty}\frac{h^0(X,mD)}{m^n/n!}\in\mathbb{R}_{\geq 0},\]
is a numerical invariant, which thus descends to a function defined on the N\'eron-Severi group $N^1(X)$ of divisors modulo numerical equivalence. The limsup is actually a limit, and divisors with positive volume are called big. The volume of a nef divisor $D$ is simply equal to its selfintersection $(D^n)$.

The volume function is homogeneous of degree $n$, so its definition can naturally be extended to $\mathbb{Q}$-divisors, and Lazarsfeld \cite[Theorem 2.2.44]{pag1} showed that the resulting volume function on $N^1(X)_\mathbb{Q}$ is locally Lipschitz, hence it extends to a continuous function
\[\vol:N^1(X)_\mathbb{R}\to \mathbb{R}_{\geq 0},\]
on the real N\'eron-Severi group of $\mathbb{R}$-divisors modulo numerical equivalence.

\subsubsection{Behavior at the pseudoeffective boundary}
The volume function is strictly positive precisely on the big cone $\mathrm{Big}(X)\subset N^1(X)_\mathbb{R}$, whose closure is the pseudoeffective cone $\ov{\mathrm{Eff}}(X)$. Thus $\vol$ vanishes on the boundary of $\ov{\mathrm{Eff}}(X)$, and a natural problem (see e.g. \cite{lehmann, mccleerey}) is to understand its order of vanishing. More precisely, given $D\in\partial\ov{\mathrm{Eff}}(X)$ a pseudoeffective $\mathbb{R}$-divisor with volume zero, and given an ample divisor $A$, then $D+sA$ is big for $s>0$ and one can ask at what rate the function
$$[0,1]\ni s\mapsto\vol(D+sA),$$
approaches $0$ as $s$ decreases to $0$. It was initially believed that this function should behave asymptotically like $s^d$ for some positive integer $d$, which is what happens when $D$ is nef.  This was recently disproved by the second-named author \cite{lesieutre}, who on a Calabi--Yau $3$-fold (which is a general complete intersection of type \((1,1)\), \((1,1)\), \((2,2)\) in \(\P^3 \times \P^3\)) constructed a pseudoeffective $\mathbb{R}$-divisor $D$ such that
$$C^{-1}s^{\frac{3}{2}}\leq \vol(D+sA)\leq Cs^{\frac{3}{2}},$$
for all $s>0$ small and for some $C>0$ (see also \cite{choipark, hoffstenger, hoffstengeryanez, jiangwang, yanez} for other recent papers that apply the method of \cite{lesieutre} to other examples). He and McCleerey \cite[Questions 4.1 and 4.2]{mccleerey} then raised the question whether we always have
\begin{equation}\label{speranza1}
C^{-1}s^\alpha\leq \vol(D+sA)\leq Cs^\alpha,
\end{equation}
for some $\alpha\in\mathbb{R}_{>0}$, or at the very least whether the limit
\begin{equation}\label{speranza2}
\lim_{s\downarrow 0}\frac{\log\vol(D+sA)}{\log s},
\end{equation}
exists.

\subsection{Growth of sections}
\subsubsection{Iitaka dimension} Given a Cartier divisor $D$, a fundamental result (see e.g. \cite[Corollary 2.1.38]{pag1}) is that the growth rate of $h^0(X,mD), m\geq 1,$ is polynomial, at least for $m$ sufficiently divisible, namely either $h^0(X,mD)=0$ for all $m\geq 1$, or else there exists $C>0$ and an integer $0\leq\kappa\leq n$ (known as the Iitaka dimension of $D$), such that for all $m\geq 1$ sufficiently divisible we have
\[C^{-1}m^\kappa\leq h^0(X,mD)\leq Cm^\kappa.\]
In fact a much more precise result, due to Kaveh-Khovanskii \cite{kk}, shows that in this case the ratio $h^0(X,mD)/m^\kappa$ converges to a positive finite limit as $m$ sufficiently divisible goes to $+\infty$.
\subsubsection{$\mathbb{R}$-divisors}
If $D$ is more generally an $\mathbb{R}$-divisor, one can still consider the space of sections using round-downs, thus looking at $h^0(X,\floor{mD}).$
Nakayama \cite{nakayama} discovered that to have better numerical properties it is convenient to twist this by an ample divisor $A$, and consider the growth rate of \[h^0(X,\floor{mD}+A).\]
It was initially hoped that this growth might also be polynomial in \(m\) \cite{nakayama,lehmann,eckl}, which is indeed the case when $D$ is nef \cite[Proposition V.2.7]{nakayama}.  However, the aforementioned example by the second-named author \cite{lesieutre} shows that this is not the case in general: the pseudoeffective $\mathbb{R}$-divisor $D$ that he constructs satisfies
\[
C^{-1} m^{3/2} \leq h^0(X,\floor{mD}+A) \leq C m^{3/2},
\]
for some $C>0$ and all $m$ sufficiently large.

Although this growth rate is not polynomial, it is not too far removed, and a natural question posed in \cite[Remark 5]{lesieutre} asks whether given any pseudoeffective $\mathbb{R}$-divisor $D$ and sufficiently ample divisor $A$ we always have
\begin{equation}\label{speranza3}
C^{-1} m^{\alpha} \leq h^0(X,\floor{mD}+A) \leq C m^{\alpha},
\end{equation}
for some $\alpha\in\mathbb{R}_{>0}$, some $C>0$ and all $m$ sufficiently large.

\subsection{The main result}
Our main result shows that all the above expectations \eqref{speranza1}, \eqref{speranza2}, \eqref{speranza3} fail.

\begin{theorem}[Oscillation of volume]
\label{main}
For every $N\geq 3$ there exists a smooth Calabi--Yau $N$-fold \(X\) such that given any real number $\delta\in [1,\frac{N}{2}]$ there exists a pseudoeffective \(\R\)-divisor \(D\) with $\vol(D)=0$ so that given any sufficiently ample divisor $A$ we have
\begin{align}
\label{1}
\liminf_{s \downarrow 0} \frac{\log \vol(D+sA)}{\log s} &= \delta,\\
\label{2}
\limsup_{s \downarrow 0} \frac{\log \vol(D+sA)}{\log s} &= \frac{N}{2},\\
\label{3}
\limsup_{m \to \infty} \frac{\log h^0(X,\floor{mD}+A)}{\log m} &= N-\delta,\\
\label{4}
\liminf_{m \to \infty} \frac{\log h^0(X,\floor{mD}+A)}{\log m} &=\frac{N}{2}.
\end{align}
\end{theorem}

\begin{figure}
\centering
\includegraphics[scale=0.2]{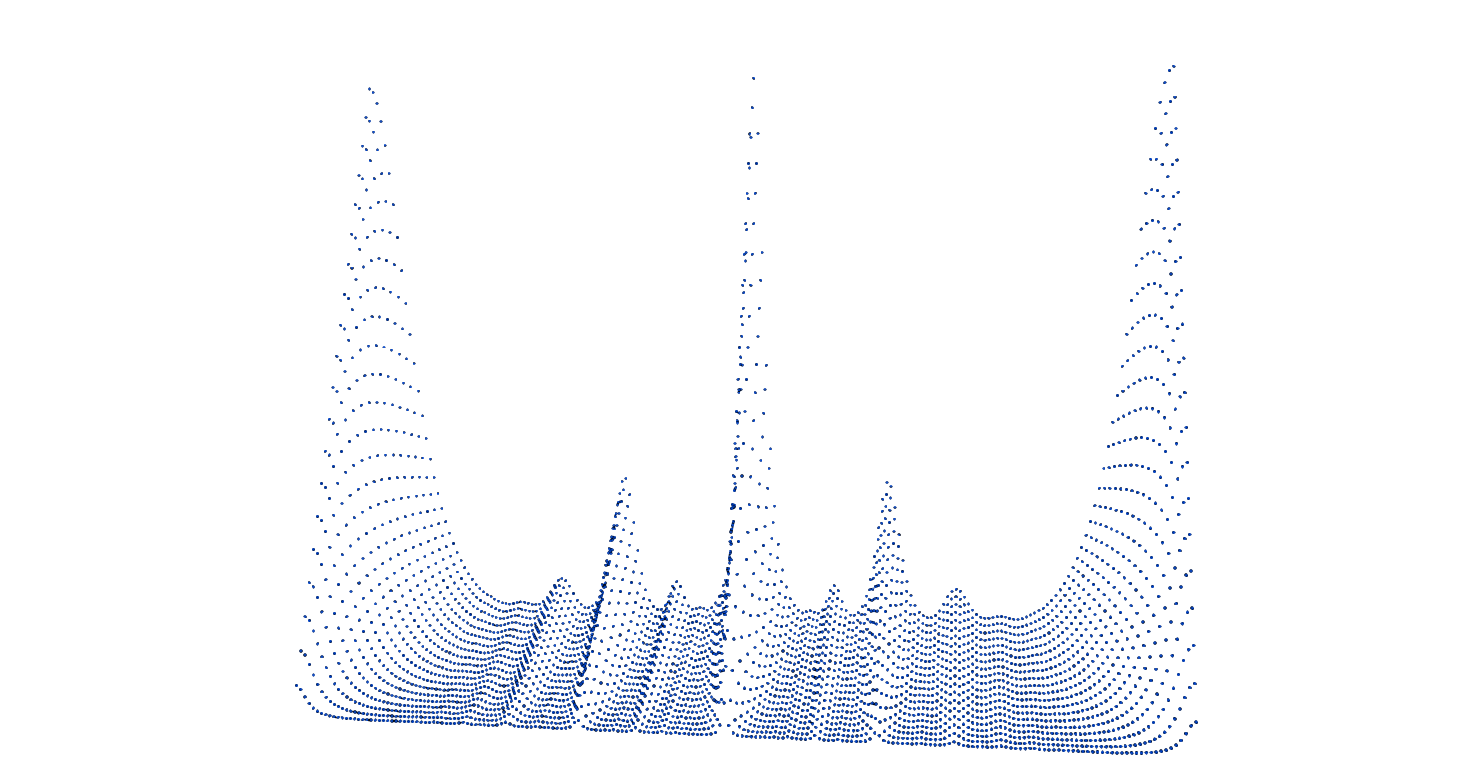}
\caption{Volume on the Wehler 3-fold}\label{fig}
\end{figure}

The manifolds $X$ in \autoref{main} are Calabi--Yau manifolds known as Wehler manifolds \cite{wehler}, which are general hypersurfaces in $(\mathbb{P}^1)^{N+1}$ of degree $(2,2,\dots,2)$. We will use some basic results about these manifolds which were proved by Cantat-Oguiso \cite{oguisocantat}.

\autoref{fig} is a numerical plot of the volume of $D_t +sA$ on a Wehler $3$-fold, where the variable $s\in [0,1]$ is plotted going ``into'' the picture, the $t$ variable is the horizontal axis, and $D_t$ parametrizes a family of pseudoeffective $\mathbb{R}$-divisors with volume zero. For the values of $t$ where $D_t$ is a $\mathbb{Q}$-divisor, $\vol(D_t+sA)$ grows linearly but with different slopes as $t$ varies, while in between these ``hills'' of linear growth one sees ``valleys'' of slower growth like $s^{\frac{3}{2}}$. In this case, the divisors $D$ constructed by \autoref{main} (with $\delta=1$) are such that $D+sA$ encounters infinitely many of such hills, and infinitely many of such valleys.

The fact that the $\limsup$ of volume (and $\liminf$ of $h^0$) appearing in \autoref{main} equal $N/2$ is, as it turns out, a general feature of the examples were consider:

\begin{theorem}[Limits for every divisor]
	\label{thm:limits_for_every_divisor}
	Suppose that $X$ is a general Wehler example Calabi--Yau $N$-fold.
	Then for every pseudoeffective $\bR$-divisor $D$ in the boundary of $\overline{\mathrm{Eff}}(X)$ which is not numerically trivial, there exists an integer $s(D)\in\{-(N-2),\dots, N-2\}$ and a real $\delta(D)\in \left[1, \tfrac{1}{2}(N-s(D))\right]$ such that
	\begin{align}
	\label{5}
	\liminf_{s \downarrow 0} \frac{\log \vol(D+sA)}{\log s} &= \delta,\\
	\label{6}
	  \limsup_{s \downarrow 0} \frac{\log \vol(D+sA)}{\log s} &= \frac{N-s(D)}{2},\\
	\label{7}
	\limsup_{m \to \infty} \frac{\log h^0(X,\floor{mD}+A)}{\log m} &= N-\delta,\\
	\label{8}
	\liminf_{m \to \infty} \frac{\log h^0(X,\floor{mD}+A)}{\log m} &=\frac{N+s(D)}{2}.
	\end{align}
\end{theorem}
This result is a paraphrasing of \autoref{thm:volume_of_points} below, combined with the estimates in \autoref{sec:from_volume_to_sections} allowing us to analyze the behavior of $h^0$ in terms of $\vol$.

\subsection{Regularity of the volume function}
\subsubsection{$C^1$ differentiability}
We now discuss a consequence of \autoref{main} to the regularity of the volume function on the pseudoeffective cone of $X$. As mentioned earlier, in \cite[Theorem 2.2.44]{pag1} Lazarsfeld proved that the volume function $\vol:N^1(X)_{\mathbb{R}}\to \mathbb{R}_{\geq 0}$ is locally Lipschitz continuous, strictly positive in the big cone $\mathrm{Big}(X)$ and zero on its complement. More regularity was obtained by Boucksom-Favre-Jonsson \cite{bfj} and Lazarsfeld-Musta\c{t}\u{a} \cite{lazarsfeldmustata} who proved that $\vol$ is $C^1$ differentiable on the big cone, and in general not $C^2$ \cite[Example 2.2.46]{pag1}. The same $C^1$ differentiability statement was proved  by Witt Nystr\"om \cite{wittnystrom} for the volume function on $H^{1,1}(X,\mathbb{R})\supset N^1(X)_{\mathbb{R}}$ (as defined by Boucksom \cite{boucksom} using analysis) in the cone of big $(1,1)$-classes, and it is conjectured to hold for general compact K\"ahler manifolds \cite{BDPP}.
\subsubsection{Higher regularity}
The question of further regularity of the volume function has been highlighted for example in \cite[Remark 2.2.52]{pag1} and \cite[Conjecture 2.18]{ELMNP}. In private conversations with the third-named author in 2015, Lazarsfeld asked whether $\vol$ enjoys better regularity at the boundary of the pseudoeffective cone, and more precisely whether given any pseudoeffective $\mathbb{R}$-divisor $D$ with $\vol(D)=0$ and any ample $A$ the function $s\mapsto\vol(D+sA)$ is smooth or even real analytic for $s\in [0,\ve)$ for some $\ve>0$ (which is of course the case when $D$ is nef). As a corollary to \autoref{main} we show that this is not the case, and find its sharp regularity:

\begin{corollary}\label{coro}
Let $X$ be as in \autoref{main}, and let $D$ be a pseudoeffective $\mathbb{R}$-divisor given by that theorem setting $\delta=1$. Then the function
$$s\mapsto \vol(D+sA),$$
defined for $s\in [0,1]$ is $C^1$ differentiable on $[0,1]$ but it is not $C^{1,\alpha}$ on $[0,\ve)$ for any $\ve>0$ and any $\alpha>0$.
\end{corollary}
Of course this implies that $\vol$ is also not $C^{1,\alpha}$ at (the numerical class of) $D$ when viewed as a function on $\ov{\mathrm{Eff}}(X)$ or on
all of $N^1(X)_{\mathbb{R}}$.

In our result the lack of higher regularity happens at the boundary of the big cone, so it is natural to ask about the optimal regularity (in general) of the volume function inside the big cone. As mentioned above, $\vol$ is known to be $C^1$ in $\mathrm{Big}(X)$, and on the blowup of $\mathbb{P}^2$ at a point it is $C^{1,1}$ in $\mathrm{Big}(X)$ but not $C^2$, see \cite[Example 2.2.46]{pag1}.

\begin{question}\label{quest}
Is there $0<\alpha\leq 1$ such that for any smooth projective variety $X$, the volume function is $C^{1,\alpha}_{\rm loc}$ on $\mathrm{Big}(X)$? If yes, what is the supremum of such $\alpha$?
\end{question}

As far as we are aware, in all the known examples where $\vol$ can be computed explicitly, it is always piecewise analytic (and often piecewise polynomial, e.g. on toric manifolds and on surfaces), see \cite{BKS} and the discussion in \cite[Conjecture 2.18]{ELMNP}. Since a piecewise analytic $C^1$ function is $C^{1,1}_{\rm loc}$, this suggests that the answer to \autoref{quest} may be affirmative, with $\alpha=1$.\footnote{Junyu Cao and the third-named author have recently confirmed \autoref{quest} with $\alpha=1$. The details will appear elsewhere.}

\subsection{Numerical dimensions}
We discuss next some applications of \autoref{main} to the study of notions of numerical dimensions of pseudoeffective $\mathbb{R}$-divisors. For a nef $\mathbb{R}$-divisor $D$ the classical notion of numerical dimension is
\begin{equation}\label{num}
\nu(D)=\max\{k\in\mathbb{N}\ |\ (D^k)\neq 0 \text{ in }H^{k,k}(X,\mathbb{R})\},
\end{equation}
and a natural extension of \eqref{num} to general pseudoeffective $\mathbb{R}$-divisors was  proposed by \cite{BDPP}:
\[\nu_{\rm BDPP}(D)=\max\{k\in\mathbb{N}\ |\ \langle D^k\rangle\neq 0 \text{ in }H^{k,k}(X,\mathbb{R})\},\]
where $\langle\cdot\rangle$ denotes the positive intersection product \cite{BDPP, bfj}, which they also show agrees with \eqref{num} when $D$ is nef. However, as observed by Nakayama \cite{nakayama}, there are a number of other possible definitions that one can propose.
\subsubsection{Nakayama's $\kappa_\sigma$}
We recall here Nakayama's definitions \cite[V.2.5]{nakayama}: given a pseudoeffective $\mathbb{R}$-divisor $D$ and a sufficiently ample divisor $A$ he defined
$$\kappa_\sigma(D)=\max\left\{k\in\mathbb{N}\ \bigg|\ \limsup_{m\to+\infty}\frac{h^0(X,\floor{mD}+A)}{m^k}>0\right\},$$
$$\kappa^-_\sigma(D)=\max\left\{k\in\mathbb{N}\ \bigg|\ \liminf_{m\to+\infty}\frac{h^0(X,\floor{mD}+A)}{m^k}>0\right\},$$
$$\kappa^+_\sigma(D)=\min\left\{k\in\mathbb{N}\ \bigg|\ \limsup_{m\to+\infty}\frac{h^0(X,\floor{mD}+A)}{m^k}<\infty\right\}.$$
Clearly $\kappa^-_\sigma(D)\leq\kappa_\sigma(D)\leq \kappa^+_\sigma(D)$. Nakayama showed \cite[V.2.7]{nakayama} that these notions are all equal when $D$ is nef, and asked whether these are all equal in general \cite[Problem, p.184]{nakayama}.
The aforementioned example by the second-named author \cite{lesieutre} satisfies $\kappa^-_\sigma(D)=\kappa_\sigma(D)=1,\kappa^+_\sigma(D)=2,$ so $\kappa^+_\sigma$ is not necessarily equal to $\kappa_\sigma$.
\subsubsection{Real number extensions}
Furthermore, one can define real numbers $\kappa^{\mathbb{R},-}_\sigma(D)\leq\kappa^{\mathbb{R}}_\sigma(D)\leq \kappa^{\mathbb{R},+}_\sigma(D)$ with the same definitions as above but allowing $k\in\mathbb{R}$ and replacing max/min with sup/inf (see e.g. \cite{lesieutre,choipark}). Generalizing Nakayama's question, it is natural to then ask whether one always has
$$\kappa^{\mathbb{R},-}_\sigma(D)=\kappa^{\mathbb{R}}_\sigma(D)=\kappa^{\mathbb{R},+}_\sigma(D).$$
This does hold in the second-named author's example \cite{lesieutre}, where all these numbers are equal to $\frac{3}{2}$.

Lastly, following \cite{lehmann, eckl} we define one more type of numerical dimension by
$$\nu_{\rm Vol}(D)=\max\{k\in\mathbb{N}\ |\ \exists\,c>0\text{ s.t. }\vol(D+tA)\geq ct^{n-k},\,\text{for all }t>0\}.$$
In all previously known examples it was the case that $$\nu_{\rm vol}(D)=\kappa_\sigma(D),$$ and the second-named author asked in \cite[Remark 8]{lesieutre} whether this always holds. We can also ask a similar question for the ``real numbers'' versions of these (where $\nu^{\mathbb{R}}_{\rm vol}(D)$ is defined in the obvious way).

Our examples show that in general $\kappa^-_\sigma$ is not equal to $\kappa_\sigma$ either, that $\nu_{\rm vol}\neq \kappa_\sigma$,  and similarly for the ``real numbers'' versions of these:

\begin{corollary}\label{kappa}
Let \(X\) and $D$ be as in \autoref{coro}. Then we have
\begin{equation}\label{agogn1}
\kappa^-_\sigma(D)=\left\lfloor\frac{N}{2}\right\rfloor,\quad \kappa_\sigma(D)\geq N-2,\quad \kappa_\sigma^+(D)=N-1,
\end{equation}
\begin{equation}\label{agogn2}
\kappa^{\mathbb{R},-}_\sigma(D)=\frac{N}{2}, \quad \kappa^{\mathbb{R}}_\sigma(D)=\kappa_\sigma^{\mathbb{R},+}(D)=N-1,
\end{equation}
\begin{equation}\label{agogn3}
\nu_{\rm vol}(D)=\left\lfloor\frac{N}{2}\right\rfloor, \quad \nu^{\mathbb{R}}_{\rm vol}(D)=\frac{N}{2}.
\end{equation}
\end{corollary}

The reader is referred to \cite{choipark} for a recent overview of general inequalities relating these (and more) different notions of numerical dimensions, including examples (based on \cite{lesieutre}) showing that $\nu_{\rm BDPP}$ is also in general different from $\nu_{\rm vol}$ and from $\kappa_\sigma$ \cite[Example 4.2]{choipark}.

\subsubsection{A conjecture of Fujino}
On a related note, motivated by the second-named author's results in \cite{lesieutre}, Fujino posed the following conjecture \cite[Conjecture 1.4]{fujino}:

\begin{conjecture}\label{cojone}
Let $X$ be a smooth projective variety over $\mathbb{C}$, and $D$ a pseudoeffective $\mathbb{R}$-divisor on $X$. Then there exist $m_0\geq 1, c>0$ and an ample divisor $A$ such that for all $m$ sufficiently large we have
\begin{equation}\label{pirla}
h^0(X,\floor{mm_0 D}+A)\geq cm^{\kappa_\sigma(D)}.
\end{equation}
\end{conjecture}

Again as a corollary of \autoref{main} we will see that this conjecture fails:

\begin{corollary}\label{scemo}
\autoref{cojone} does not hold on any Wehler $N$-fold, $N\geq 5$.
\end{corollary}

\subsection{Idea of proof}

\subsubsection*{Hyperbolic manifolds}
The analysis of the volume function is facilitated by the fact that it is invariant under pseudo-automorphisms of $X$ and Wehler examples admit a large group of these, namely $\PsAut(X)$ is equal to the free product $W$ of $N+1$ copies of $\bZ/2$.
It was observed by Cantat--Oguiso \cite{oguisocantat} that $W$ preserves a quadratic form of signature $(1,N)$ on the \Neron--Severi group of $X$.
One can then consider the action of $W$ on the corresponding hyperbolic space $\bH^N$ and form the quotient, a hyperbolic manifold $\cM:=\bH^N/W$ which is of infinite volume as soon as $N\geq 3$.
The group and manifold are nevertheless geometrically finite and we have a ``convex core'' $\cM^{cc}\subset \cM$, see \autoref{ssec:some_further_cones_and_some_dynamics} for these notions.
Furthermore, $\cM^{cc}$ is not compact but rather admits a decomposition into a compact subset and finitely many cusps, see \autoref{ssec:calculations_in_the_cusp}.
Additionally, the complement $\cM\setminus \cM^{cc}$ consists of infinite-volume ``funnels''.

\subsubsection*{Hyperbolic geodesics}
Similarly to \cite{FT}, our analysis proceeds by associating to a path of divisors $D+sA$ with $s\in[0,1]$ a hyperbolic geodesic $\gamma_t$ on the manifold $\cM$, and vice-versa.
When $D$ is on the boundary, this leads to an infinite geodesic ray so $t\in [t_0,+\infty)$.
The behavior of the volume function is then naturally related to the behavior of the geodesic: large values of the volume correspond to deep excursions into a cusp.
With this principle understood, it is immediate to produce hyperbolic geodesics along which the volume oscillates as desired in \autoref{main}.

Furthermore, every infinite to one side geodesic ray in $\cM^{cc}$ is either recurrent (returns infinitely often to a compact part) or eventually goes out into a cusp.
This dichotomy is, in particular, responsible for \autoref{thm:limits_for_every_divisor}: when the geodesic goes into a cusp, the calculation of volume is explicit, while in the recurrent case it always achieves low values along a subsequence.

\subsubsection*{Inductive structure}
The above outline applies to divisors $D$ such that the associated hyperbolic geodesic is confined to $\cM^{cc}$.
To analyze an arbitrary $D$, we introduce in \autoref{thm:divergent_recurrent_decomposition_of_limit_points} a ``divergent--recurrent'' decomposition of $D=D_d+D_r$ that reduces the analysis of the volume function along $D+sA$ to that along $D_r + sA$, which in turn is related to a recurrent geodesic ray in $\cM^{cc}$.

Specifically, $D_r$ is a divisor whose associated hyperbolic geodesic ray stays in $\cM^{cc}$, while $D_d$ is the $W$-image of a linear combination of hyperplane sections.
The geodesic ray associated to $D_d$ leaves $\cM^{cc}$ and goes out into a ``funnel'' of $\cM\setminus \cM^{cc}$, but along which the behavior of volume (and its influence on the volume of $D$) is analyzed by an elementary calculation.
The precise statement is contained in \autoref{thm:volume_of_points}.

\subsubsection*{Scaling}
By construction, the volume function of an $N$-fold is $N$-homogeneous, namely $\vol(\lambda B)=\lambda^N \vol(B)$.
When passing between linear paths $D+sA$ and hyperbolic geodesics $\gamma_t$, the scaling of volume has to be taken into account and it is interesting to note that $N/2$ appears as the natural symmetry point.

\subsubsection*{Cone Conjecture}
It is expected \cite{Morrison_Compactifications-of-moduli-spaces-inspired1993,Kawamata1997_On-the-cone-of-divisors-of-Calabi-Yau-fiber-spaces,LOP} that Calabi--Yau manifolds have large groups of (pseudo-)automorphisms which control their pseudoeffective cones, in particular have a finite-sided polyhedral fundamental domain.
In cases where such a property holds, we expect that the methods of this paper can be used to analyze the behavior of volume for all points on the boundary as well.

\subsection*{Acknowledgments}
We are grateful to Yuan Fu, Heming Liu and Junsheng Shi for helping with numerical investigations on the volume on Wehler $3$-folds, and for \autoref{fig}, and to the referee for useful suggestions. The first-named author is grateful to Jeff Danciger for conversations related to the convex core of hyperbolic manifolds associated to Coxeter groups.  The second-named is grateful to Serge Cantat for a useful discussion.
The third-named author is also grateful to Rob Lazarsfeld for inspiring conversations. The authors were partially supported by NSF grants DMS-2142966 (JL), DMS-2005470, DMS-2305394 (SF), DMS-2231783, DMS-2404599 (VT), and a Sloan fellowship (JL).

%%%%%%%%%%%%%%%%%%%%%%%%%%%%%%%%%%%%%%%%%%%%%%%%%%%%%%%%%%%%%%%%%%%%%%%%%%%%%%%
%%% 				Start of Section: Notation and the basic objects
%%%%%%%%%%%%%%%%%%%%%%%%%%%%%%%%%%%%%%%%%%%%%%%%%%%%%%%%%%%%%%%%%%%%%%%%%%%%%%%

\section{Notation and the basic objects}
	\label{sec:notation_and_the_basic_objects}

\paragraph{Outline of section}
In this section we introduce some notation for the basic objects of interest, which include various positive cones in the cohomology of a Wehler $N$-fold, and give a decomposition of points on the boundary of the big cone as a sum of a ``divergent'' part and a ``recurrent'' part in \autoref{thm:divergent_recurrent_decomposition_of_limit_points}.

\subsection{Wehler manifolds}\label{ssec:wehler}
Throughout this paper, $X$ will denote a Wehler $N$-fold, $N\geq 2$, which is a general hypersurface in $(\mathbb{P}^1)^{N+1}$ of degree $(2,2,\dots,2)$. By adjunction we see that $X$ is Calabi--Yau. It is important to keep in mind that the case $N=2$, where $X$ is then a $K3$ surface, behaves differently from the case $N\geq 3$.

$X$ has $N+1$ projections to $(\mathbb{P}^1)^{N}$, by forgetting one of the factors, which are all ramified $2:1$ covers, and we let $\sigma_i, 0\leq i\leq N,$ be the birational covering involutions of $X$. Since $X$ is Calabi--Yau, these are pseudo-automorphisms of $X$, i.e. isomorphisms in codimension $1$. In particular, each $\sigma_i$ acts on $N^1(X)_{\mathbb{R}}$ by pulling back divisors. Furthermore, thanks to \cite[Theorem 1.3]{oguisocantat}, the birational automorphism group of $X$ is isomorphic to the free product $W$ of $N+1$ copies of $\mathbb{Z}/2$, generated by $\sigma_i, 0\leq i\leq N$.  The N\'eron-Severi group $N^1(X)$ is of rank $N+1$, generated by the restriction of the $N+1$ hyperplane classes from the $\mathbb{P}^1$ factors.

%%=============================================================================
%%					start of subsec: Bases and some cones

\subsection{Bases and some cones}
	\label{ssec:bases_and_normalizations_for_the_quadratic_form}

\subsubsection{Setup}
	\label{sssec:setup_bases_and_normalizations_for_the_quadratic_form}
We take $N\geq 2$ for this discussion, with the caveat that when $N=2$ some of the faces in \autoref{prop:convex_core_fundamental_domain_properties} consist of one vector only and the algebro-geometric interpretation of the cones in \autoref{sssec:algebro_geometric_dictionary} is a bit different as well.
Note that our notation follows Cantat--Oguiso \cite{oguisocantat} but uses $\omega_i$ instead of the $c_i$ for the dual basis.

\subsubsection{Basis of hyperplane sections}
	\label{sssec:basis_of_hyperplane_sections}
Let $\omega_0,\ldots, \omega_N\in N^1(X)$ be the basis given by the hyperplane sections.
The quadratic form introduced by Cantat--Oguiso is, up to dividing by ${2(N-1)}$, given by the following bilinear form (see \cite[\S2.2.3]{oguisocantat}):
\begin{align}
	\label{eqn:quadraticform_omega_basis}
	\begin{cases}
	\ip{\omega_i, \omega_i} = - (N-2) & \forall i\\
	\ip{\omega_i,\omega_j}  = 1       & \forall i\neq j
	\end{cases}
\end{align}
It is nondegenerate of signature $(1,N)$.

\subsubsection{Duals to hyperplane sections}
	\label{sssec:duals_to_hyperplane_sections}
Let $\alpha_0,\ldots,\alpha_N\in N^1(X)$ be defined by
\begin{align}
	\label{eqn:c_in_alpha_basis}
	\begin{split}
	\alpha_i & := \omega_0 + \cdots + \omega_{i-1} - \omega_i + \omega_{i+1} + \cdots + \omega_N\\
	u & := \omega_0 + \phantom{ \cdots + \omega_{i-1} +} \cdots\phantom{+ \omega_i - \omega_{i+1} + } + \omega_N\\
	& \text{ so that:}\\
	\alpha_i & = u - 2\omega_i
	\end{split}
\end{align}
Then we have $\ip{\omega_i,\alpha_j}=2(N-1)\delta_{ij}$ and $\{\alpha_i\}$ are (up to a factor) the basis dual with respect to the bilinear form.
We have the basic identities:
\begin{align*}
	\ip{\alpha_i,\alpha_i}=-2(N-1)<0 \text{ and } \ip{\alpha_i,\alpha_j}=2(N-1)>0
	\quad i\neq j.
\end{align*}
Furthermore $\ip{u,u}=2(N+1)$ while $\ip{u,\omega_i}=2$.

\subsubsection{The reflections}
	\label{sssec:the_reflections}
We will make use of the following involutions on $N^1(X)$, which clearly preserve the bilinear form as they are given by orthogonal reflections in $\alpha_i$:
\begin{align}
	\label{eqn:formula_for_tau_reflections}
	\begin{split}
	\sigma_i(v) & := v + \frac{\ip{v,\alpha_i}}{N-1}\alpha_i \text{ or explicitly:}\\
	\sigma_i(\omega_j) & =
	\begin{cases}
		\omega_j &\text{if }i\neq j\\
		\omega_i + 2 \alpha_i & \text{if }i=j
	\end{cases}
	\end{split}
\end{align}
Let $W\subset \bbO(N^1(X))$ be the group generated by the reflections $\{\sigma_i\}$.
These $N+1$ generators satisfy no relations other than $\sigma_i^2=\id$ so we have that $W\isom \bZ/2 * \cdots * \bZ/2$.
These reflections coincide with the action by pullback of the covering birational involutions of $X$ discussed in \autoref{ssec:wehler}, see \cite[Theorem 3.3]{oguisocantat}.

\subsubsection{Cones and Duality}
	\label{sssec:cones_and_duality}
We will be interested in several convex cones inside $N^1(X)$; a cone will refer to any subset invariant under scaling by $\bR_{>0}$.
Our cones will always be \emph{properly} convex, meaning that they are convex and do not contain any line through the origin.

Recall that for any cone $\cL\subset N^1(X)$, its dual $\cL^{*}\subset N^1(X)$ with respect to the bilinear form consists of all vectors $v$ such that $\ip{v,t}\geq 0, \forall t\in \cL$.
Note also that with this definition, the dual of a cone is always closed, even if the original wasn't.
The closure of a cone $\cL$ (or any other subset) will be denoted by $\ov{\cL}$. The Hahn-Banach theorem implies that $\cL^{**}=\ov{\cL}$.

\subsubsection{The basic simplicial cones}
	\label{sssec:the_basic_simplicial_cones}
Denote by $\cA$ the open cone in $N^1(X)$ spanned by the $\{\omega_i\}$, i.e. consisting of all vectors $\sum_i x_i \omega_i$ with $x_i> 0$.
Let also ${\cC}$ denote the open cone in $N^1(X)$ spanned by the classes $\{\alpha_i\}$, i.e. consisting of all vectors $\sum_i x_i \alpha_i$ with $x_i> 0$.
Then its closure $\ov{\cC}$ is dual to $\ov{\cA}$.

\subsubsection{The Tits and minimal cones}
	\label{sssec:the_tits_and_minimal_cones}
Consider the cone $\cT:=W\cdot \ov{\cA}$, which is neither open nor closed.
In the literature on Coxeter groups, this is also called the \emph{Tits cone}, hence the reason for denoting it by $\cT$.
The closure of the Tits cone $\ov{\cT}$ is a properly convex, closed, and $W$-invariant cone and furthermore \emph{maximal} with these properties by \cite[Prop.~4.1]{DancigerGueritaud_Convex-cocompactness-for-Coxeter-groups2023}.

Conversely, let $\ov{\cC\cH}$ denote the dual of $\cT$.
This is necessarily a closed cone, and is furthermore minimal among closed $W$-invariant properly convex cones (since $\ov{\cT}$ is maximal) and hence contained in $\ov{\cT}$.

\subsubsection{Algebro-Geometric dictionary}
	\label{sssec:algebro_geometric_dictionary}
Let us recall (see \autoref{ssec:wehler}) that if $X$ is a Wehler $N$-fold, then the reflections $\sigma_i$ are given by the action in $N^1(X)$ of pseudo-automorphisms (denoted by the same symbols) of $X$; we let $\PsAut(X)$ denote the group of pseudo-automorphisms of $X$.

We also have the following interpretations of our cones in terms of the geometry of $X$, assuming $N\geq 3$ (see \cite[\S 3.4]{oguisocantat} for the case $N=2$):
\begin{itemize}
	\item The open cone $\cA$ is the \emph{ample cone} of $X$ while its closure $\ov{\cA}$ is the \emph{nef cone} of $X$, by \cite[Thm.~1.3, Thm.~3.1(2)]{oguisocantat}.
	\item The cone $\cT$ is the movable effective cone of $X$, by \cite[Thm.~1.3(3), Thm.~4.1(3)]{oguisocantat}.
\end{itemize}

\begin{proposition}[The pseudoeffective and big cones]
	\label{prop:characterizing_the_pseudoeffective_cone}
	\leavevmode
	\begin{enumerate}
		\item The closure $\ov{\cT}$ coincides with the pseudoeffective cone of $X$.
		\item The interior of $\cT$ coincides with the big cone of $X$.
	\end{enumerate}
\end{proposition}
\begin{proof}
	Both the pseudoeffective and the big cone are $\PsAut(X)$-invariant, and the first is closed while the second is open.
	Furthermore the first contains $\ov{\cA}$ while the second one contains $\cA$.
	It follows that the pseudoeffective cone contains $\ov{\cT}$ while the big cone contains the interior of $\cT$.

	Since $\ov{\cT}$ is the maximal properly convex closed $W$-invariant cone (by \autoref{sssec:the_tits_and_minimal_cones}) it follows that the preceding containments are in fact equalities.
\end{proof}

Let us end by remarking that the classes $\alpha_i$ from \autoref{sssec:duals_to_hyperplane_sections} do not have any (apparent to us) meaning in terms of the geometry of $X$.
Specifically, since $\{\alpha_i\}$ is the basis dual to $\{\omega_i\}$ for the ad hoc quadratic form, it would be interesting to understand if there is some natural way to associate a curve class to each $\omega_i$ so as to define the quadratic form geometrically.

% \subsubsection{Fundamental domain properties}
% 	\label{sssec:fundamental_domain_properties}

%%					end of subsec: Bases and some cones
%%=============================================================================

%%=============================================================================
%%					start of subsec: Some further cones and some dynamics

\subsection{Some further cones and some dynamics}
	\label{ssec:some_further_cones_and_some_dynamics}

% \subsubsection{Setup}
% 	\label{sssec:setup_some_further_cones_and_some_dynamics}
We keep the notation from the preceding paragraph.

\subsubsection{Projectivization}
	\label{sssec:projectivization}
For a subset $Q\subset N^1(X)$ we will denote by $[Q]\subset \bP N^1(X)$ its image inside the real projectivized \Neron--Severi group.
When dealing with cones, we will work exclusively with properly convex ones, so that the cone can be recovered from its projectivization uniquely, up to a choice of connected component which will always be the one intersecting the big cone.

Additionally, given a set $R\subset \bP N^1(X)$, we will denote by $\wtilde{R}$ its preimage intersected with the closure of the Tits cone.
In particular, in all the situations we consider we'll have for $Q\subset N^1(X)$ that $Q=\wtilde{[Q]}$.
When speaking of the \emph{convex hull} of a set $X\subset \bP N^1(X)$, we will always understand this as the convex hull of $\wtilde{X}$, projectivized.

\subsubsection{Hyperbolic space}
	\label{sssec:hyperbolic_space}
Let $\cH$ denote the connected component of vectors satisfying $\ip{v,v}>0$ that intersects $\cA$.
Then $\cH$ is an open cone, with closure denoted $\ov{\cH}$.
Denote by $\cH^{1}\subset \cH$ the hyperboloid of unit-norm vectors, it provides a model for hyperbolic $N$-space.
The projectivization $[\cH]$ is naturally in bijection with $\cH^1$.

\subsubsection{Limit set}
	\label{sssec:limit_set}
Denote by $\Lambda\subset \partial[\cH]=\left[\ov{\cH}\right]\setminus [\cH]$ the limit set of the group $W$.
By definition, $\Lambda$ is the set of accumulation points of an orbit $W\cdot [v]$ for some $[v]\in [\cH]$; it is independent of the choice of such $[v]$.
Furthermore, it is the smallest closed $W$-invariant set in $\partial [\cH]$, see \cite[\S12.2]{Ratcliffe2019_Foundations-of-hyperbolic-manifolds} for all these assertions.

\subsubsection{Convex hull of limit set}
	\label{sssec:convex_hull_of_limit_set}
Recall from \autoref{sssec:the_tits_and_minimal_cones} that $\ov{\cC\cH}$ denoted the dual of $\ov{\cT}$ and was the smallest closed properly convex $W$-invariant cone.
Furthermore, by \cite[Prop.~2.6]{DancigerGueritaud_Convex-cocompactness-for-Coxeter-groups2023} or \cite[Prop.~3.1]{Benoist2000_Automorphismes-des-cones-convexes} we have that $\ov{\cC\cH}$ is the convex hull of the (deprojectivized) limit set $\wtilde{\Lambda}$.
We will denote by $\cC\cH$ the interior of $\ov{\cC\cH}$.
Following \cite[\S12.4, p.~634]{Ratcliffe2019_Foundations-of-hyperbolic-manifolds} but with some license, we will refer to $\ov{\cC\cH}\cap \cH^1$ as the ``convex core'' of the group $W$.

Note also that we have the sequence of $W$-invariant cones $\ov{\cC\cH}\subseteq \ov{\cH}\subseteq \ov{\cT}$ which is exchanged by duality, in particular $\ov{\cH}$ is its own dual.

\subsubsection{Fundamental domains}
	\label{sssec:fundamental_domains}
We already know from \autoref{sssec:the_tits_and_minimal_cones} that $\ov{\cA}$ is a fundamental domain for $W$ acting on $\cT$.
To obtain the fundamental domain on $\cH$, we simply intersect $\ov{\cA}$ with $\cH$.

Set now $\ov{\cF}:=\ov{\cC}\cap \ov{\cA}$, where recall that $\ov{\cC}$ was the cone dual to $\ov{\cA}$.
Then by \cite[Thm.~5.2]{DancigerGueritaud_Convex-cocompactness-for-Coxeter-groups2023} we have that $\cF:=\ov{\cF}\cap \cH$ is the fundamental domain for the $W$-action on $\cC\cH$.
Note also that both of $\ov{\cC}$ and $\ov{\cA}$ are simplicial cones (i.e. the vertices give a basis of the ambient vector space).

For ease of notation, from now on we will denote
\begin{equation*}
[N]:=\{0,\dots,N\}.
\end{equation*}

\begin{proposition}[Convex core fundamental domain properties]
	\label{prop:convex_core_fundamental_domain_properties}
	Suppose that $N\geq 3$.
	Define the vectors
	\[
		\omega_{\widehat{ij}}:=
		\left(\omega_0 + \dots + \omega_N\right) - \left(\omega_i + \omega_j\right)
	\]
	for all two-element subsets $\{i,j\}\subset [N]$, for a total of $\binom{N+1}{2}$.
	\begin{enumerate}
		\item The vectors $\omega_{\widehat{ij}}$ are isotropic and span all the extremal rays of $\ov{\cF}$.
		\item The cone $\ov{\cF}$ has $2(N+1)$ faces of codimension $1$:
		\begin{align*}
			\ov{\cF}_{k} & := \bR_{\geq 0}\text{-span  of }
			\{\omega_{\widehat{ij}} \colon k\notin \{i,j\}\}
			\quad k = 0, \dots, N\\
			\ov{\cF}_{\widehat{k}} & := \bR_{\geq 0}\text{-span  of }
			\{\omega_{\widehat{ij}} \colon k\in \{i,j\}\}
			\quad k = 0, \dots, N
		\end{align*}
		\item The face $\ov{\cF}_k$ is contained in $\omega_{k}^{\perp}$ and $\ov{\cF}_{\widehat{k}}$ is contained in $\alpha_k^\perp$.

\noindent
		In particular, the reflection $\sigma_{k}$ fixes pointwise $\ov{\cF}_{\widehat{k}}$, while $\sigma_{i}$ for $i\neq k$ keeps $\ov{\cF}_{k}$ in the hyperplane $\omega_k^\perp$.
	\end{enumerate}	
\end{proposition}
\noindent We will make use of the above statements and notation when $N=2$ as well, the only caveat is that $\ov{\cF}_k$ is not a face but an extremal ray of the fundamental domain.

We will also refer to $\ov{\cF_k}$ as a \emph{divergent} face and $\ov{\cF}_{\widehat{k}}$ as \emph{recurrent} face.
This is not standard terminology, but will be useful to keep track of other structures.

\begin{proof}
	It is a direct calculation that $\omega_{\widehat{ij}}$	is isotropic.
	It also follows from the definition of $\ov{\cF}$ as the intersection of $\ov{\cC}$ and $\ov{\cA}$ that it has at most $2(N+1)$	faces of codimension $1$, as it is given by the $2(N+1)$ equations
	\[
		\{v\colon \ip{\alpha_i,v}\geq 0, \ip{\omega_i,v}\geq 0\quad i=0,\dots, N\}.
	\]
	It is also immediate to compute from the definitions that the vectors $\omega_{\widehat{ij}}$ belong to the listed hyperplanes:
	\begin{align*}
		\ip{\omega_{\widehat{ij}},\omega_k} & = 0 \quad \text{if }k\notin\{i,j\}\\
		\ip{\omega_{\widehat{ij}}, \alpha_k} & = 0 \quad \text{if }k\in\{i,j\}
	\end{align*}
	The claim about the action of $\sigma_i$ on the faces, and their hyperplanes, is again immediate from the definition of the reflections.

	It remains to verify that the $\omega_{\widehat{ij}}$ indeed span extremal rays of $\ov{\cF}$, and that there are no others.
	First, observe that $\omega_{\widehat{ij}}$ belongs to the intersection of the $N+1$ hyperplanes $\omega_k^\perp$ for $k\in [N]\setminus \{i,j\}$ and $\alpha_{i}^\perp,\alpha_j^\perp$, and is in the strictly positive halfspace determined by the remaining $N+1$ equations.
	This implies that indeed the ray it spans is extremal.

	To verify that there are no other extremal rays, suppose there was one spanned by a vector $e\neq 0$.
	If $e\in \omega_k^{\perp}$ for some $k$ then we can proceed inductively (see also \autoref{ssec:inductive_structure_of_boundary} for related arguments) and conclude that it was one of the $\omega_{\widehat{ij}}$.
	If on the other hand it did not belong to any $\omega_k^\perp$, then it must be in the intersection of at least $N$ hyperplanes $\alpha_i^\perp$, which would imply it agrees with some $\omega_j$ which is a contradiction.
\end{proof}

\subsubsection{Parabolic vectors}
	\label{sssec:parabolic_vectors}
Any vector that is proportional to a vector in the $W$-orbit of any of the $\omega_{\widehat{ij}}$ will be called a \emph{parabolic vector}.
Note that a parabolic vector belongs to $W\cdot \ov{\cA}$, so by the action of a single element $w\in W$ it can be made to belong to $\ov{\cA}$.

On the other hand, it also belongs to the limit set $\wtilde{\Lambda}$, since $\omega_{\widehat{ij}}$ is fixed by the two reflections $\sigma_i,\sigma_j$ whose composition is a unipotent transformation with attracting ray spanned by $\omega_{\widehat{ij}}$, i.e. for any $v\notin \omega_{\widehat{ij}}^\perp$ we have that
\[
	(\sigma_i\sigma_j)^{n} [v]  \xrightarrow{n\to +\infty} [\omega_{\widehat{ij}}] \text{ in }\bP N^1(X).
\]

\subsubsection{Geometric finiteness and recurrence}
	\label{sssec:geometric_finiteness_and_recurrence}
We record a few basic properties of the group $W$ and its limit set $\Lambda$.
First, the action of $W$ on hyperbolic space $\cH^1$ has a finite-sided polyhedral fundamental domain, namely $\ov{\cA}\cap \cH^1$.
Therefore, $W$ is geometrically finite in the sense of \cite[\S12.4, pg.~635]{Ratcliffe2019_Foundations-of-hyperbolic-manifolds}.
In particular, every point in the limit set is either conical, or bounded parabolic by \cite[Thm.~12.4.5]{Ratcliffe2019_Foundations-of-hyperbolic-manifolds}.

A boundary limit point $p$ is \emph{conical} if for every point $a\in \cH^1$, there exists a compact set $N_{a,p}\subset \cH^1$ such that the hyperbolic geodesic ray connecting $a$ to $p$ enters infinitely often the $W$-orbit of $N_{a,p}$, see \cite[\S12.3, Definition]{Ratcliffe2019_Foundations-of-hyperbolic-manifolds}.
It is elementary to check that if we fix a compact set of the convex core $K\subset \cH^1\cap \ov{\cC\cH}$, then there exists a compact set $N_K$ also in the convex core, such that $N_K$ can be used in the definition of a conical limit point, taking any $a\in K$.

Finally, we recall that a \emph{bounded parabolic} limit point $p$ (see \cite[\S12.3, pg.~618]{Ratcliffe2019_Foundations-of-hyperbolic-manifolds}) is one for which the quotient $\left(\Lambda\setminus \{p\}\right)/\Stab_W(p)$ is compact.

Let us also note that by \cite[Thm.~12.4.4]{Ratcliffe2019_Foundations-of-hyperbolic-manifolds}, we have that the only rays in $\ov{\cA}\cap \wtilde{\Lambda}$ are precisely the parabolic ones, namely the rays spanned by the $\omega_{\widehat{ij}}$ from \autoref{prop:convex_core_fundamental_domain_properties}.

\subsubsection{The structure at the cusp}
	\label{sssec:the_structure_at_the_cusp}
For future use, we describe also the structure of $\cF$ near the ray spanned by $\omegaij$.
The faces to which it belongs are
\begin{itemize}
	\item The $(N-1)$ divergent faces $\cF_k$ with $k\in [N]\setminus \{i,j\}$.
	\item The $2$ recurrent faces $\cF_{\hat{i}},\cF_{\hat{j}}$.
\end{itemize}
Note, in particular, that $\cF$ is not simplicial at the ray of $\omegaij$, since $N+1$ faces meet there instead of $N$.
Note also that when $N=2$ the one divergent face is in fact a ray.

To determine which other rays spanned by $\omega_{\widehat{k\ell}}$ are adjacent to the ray of $\omega_{ij}$, we note that it should be on at least $N-1$ faces from the above list.
This gives the list of vectors
\[
	\omega_{\widehat{i\ell}} \text{ with } \ell\in [N]\setminus \{i,j\}
	\quad \text{ and }\quad
	\omega_{\widehat{kj}} \text{ with } k\in [N]\setminus \{i,j\}.
\]

%%					end of subsec: Some further cones and some dynamics
%%=============================================================================

%%=============================================================================
%%					start of subsec: Inductive structure of boundary

\subsection{Inductive structure of boundary}
	\label{ssec:inductive_structure_of_boundary}

Let us fix a point $a\in \cF$, for definiteness we could take $a = \sum_i \omega_i$ for example.
Our goal is to now give a classification of points $p$ in $\partial\ov{\cT}$ according to how the segment $[a,p]$ crosses different $W$-translates of the fundamental domains $\ov{\cA}$.

\subsubsection{Setup}
	\label{sssec:setup_inductive_structure_of_boundary}
Fix a subset $S\subset [N]$.
Associated to it we can define
\begin{align*}
	\cA_{S}& :=\left\{ \sum_{s\in S} x_s \omega_s\colon x_s>0\quad \forall s\in S \right\}\\
	W_{S} & := \ip{\sigma_s\colon \forall s\in S} \subset W
\end{align*}
So $\cA_{S}$ is the interior of a facet of $\ov{\cA}$, while $W_S$ is the subgroup of $W$ generated by the reflections in $S$.
Additionally, we will denote by $\Lambda_S\subset \partial \left[\cH\right]$ the limit set of $W_S$ in the sense of \autoref{sssec:limit_set}, and its lift to the vector space will be denoted $\wtilde{\Lambda}_S$.

If we abbreviate the complement as $S^c:=[N]\setminus S$ then the stabilizer of $\cA_S$ is $W_{S^c}$.
We can now state the main result:

\begin{theorem}[Divergent--recurrent decomposition of limit points]
	\label{thm:divergent_recurrent_decomposition_of_limit_points}
	Suppose $N\geq 2$.
	For every $p\in \partial \ov{\cT}$ there exists a unique decomposition
	\[
		p = p_d + p_r \quad \text{with}\quad p_d,p_r \in \partial \ov{\cT}
	\]
	characterized by the following properties.
	\begin{enumerate}
		\item There exists a $w'\in W$ (possibly $\id$) and $q_{d}'\in \cA_{S_d'}, q_r'\in \wtilde{\Lambda}_{S_r'}$ where $S_r',S_d'\subseteq [N]$ are disjoint, such that:
		\[
			p_d': = w\cdot q_d' \quad p_r' : = w\cdot q_r'\text{ and }
			p = p_d' + p_r'.
		\]		
		\item Among all choices of $w'$ satisfying the condition of the first part, pick one such that the cardinality of $S_d'$ is maximized.
		Then the resulting vectors $p_d',p_r'$ are independent of the choice of this $w'$.
	\end{enumerate}
	% \leavevmode
	Furthermore, if $S_r$ is nonempty then $|S_r|\geq 3$ and if $S_d$ is nonempty then $|S_d|\leq N-1$.
\end{theorem}
\noindent We will refer to $p_d$ and $p_r$ as the ``divergent'' and ``recurrent'' part of $p$ respectively.
With our choices, note that $W_{S_r}$ fixes $q_d$.

\begin{remark}[On uniqueness]
	\label{rmk:on_uniqueness}
	\leavevmode
	\begin{enumerate}
		\item Since the decomposition in \autoref{thm:divergent_recurrent_decomposition_of_limit_points} is unique, any other possible choice of $w\in W$ differs by a stabilizer of both $\cA_{S_d}$ and $\wtilde{\Lambda}_{S_r}$.
		Therefore, $w$ gives a unique coset in $W/W_{T}$ with $T:=[N]\setminus \left(S_d \coprod S_r\right)$.
		\item The maximality condition on $S_d$ is needed because of parabolic vectors, see \autoref{sssec:parabolic_vectors}.
		Indeed, $\omega_{\widehat{ij}}$ belongs both to the limit set and to $\cA_{[N]\setminus \{i,j\}}$.
		Its decomposition resulting from \autoref{thm:divergent_recurrent_decomposition_of_limit_points} is
		\[
			\omega_{\widehat{ij}} =
			\underbrace{\omega_{\widehat{ij}}}_{p_d} +
			\underbrace{\phantom{0}0\phantom{_{\widehat{j}} } }_{p_r}
		\]
		with vanishing recurrent part.
	\end{enumerate}
\end{remark}

\begin{remark}[On extremal rays]
	\label{rmk:on_extremal_rays}
	\autoref{thm:divergent_recurrent_decomposition_of_limit_points} provides in particular a description of the extremal rays of $\ov{\cT}$.
	Namely, we find that the extremal rays of $\ov{\cT}$ are spanned by the $\omega_i$ and their $W$-orbits, as well as the (lift of) the limit set of $W$.
	However, the decomposition provided by the Theorem gives further information about the structure of the boundary.
\end{remark}

\subsubsection{Base case: $N=2$}
	\label{sssec:base_case_n_2}
Since we will prove the statement by induction, we treat the base case first.
Note that we can apply the same reasoning as in \autoref{sssec:proof_of_thm:divergent_recurrent_decomposition_of_limit_points} below, with the observation that only the first two of the three cases can occur here.

When $N=2$ the reflection group $W$ is a lattice so we have that $\partial\ov{\cT}=\partial\ov{\cH}=\wtilde{\Lambda}$.
There are two possibilities for a point $p\in \wtilde{\Lambda}$: either it is parabolic, or not.
If $p$ is parabolic then we set $p=p_d$, in accordance with \autoref{rmk:on_uniqueness} above (then $S_d\in \{\{0,1\},\{0,2\},\{1,2\}\}$ and $S_r=\emptyset$).
If $p$ is not parabolic then we set $p=p_r$ and $S_d=\emptyset$ and $S_r=\{0,1,2\}$.

\subsubsection{Preliminaries for induction}
	\label{sssec:preliminaries_for_induction}
To describe the inductive step, we introduce some temporary notation.
Assume $N\geq 3$.
We will use the subscript $N$ for cones that occur in the \Neron--Severi group of a hypersurface of dimension $N$, and respectively subscript $N-1$ for those of one dimension lower.
In addition we will add a $'$ superscript to objects corresponding to hypersurfaces of one dimension lower to distinguish them from analogous ones in one dimension higher.
For example, $W_{[N-1]}$ is a subgroup of $W_{[N]}$ while $W_{[N-1]}'$ denotes the corresponding reflection group in one dimension lower.

We will denote the corresponding vector spaces by $V_N$ and $V_{N-1}'$ respectively and work with the conformal embedding
\begin{align*}
	\phi\colon V_{N-1}'&\into V_N\\
	\omega_{i}' &\mapsto \omega_{i} + \lambda \omega_N && i=0,\dots, N-1
\end{align*}
where $\lambda$ is determined from the condition $\phi(\omega_i')\in \omega_{N}^\perp$, i.e. $\lambda=\tfrac{1}{N-2}$.
The embedding is almost isometric -- it scales the quadratic form by $(N-1)/(N-2)$.
It is also equivariant for the action of $W'_{[N-1]}$ on $V'_{N-1}$ and the morphism $\phi_W\colon W'_{[N-1]}\toisom W_{[N-1]} \subset W_{[N]}$.

Denote by $\cT_N,\cT_{N-1}'$ the corresponding Tits cones.
With this preliminary notation, we can now state:

\begin{proposition}[Codimension 1 decomposition]
	\label{prop:codimension_1_decomposition}
	Let $\omega_N^{\leq 0}$ denote the halfspace of vectors in $V_N$ with nonpositive inner product with $\omega_N$.
	\leavevmode
	\begin{enumerate}
		\item We have that
		\[
			\ov{\cT}_{N}\cap \omega_N^\perp = \phi\left(\ov{\cT}'_{N-1}\right).
		\]
		\item Furthermore we have that
		\[
			\ov{\cT}_N\cap \omega_N^{\leq 0}
			=
			\text{Cvx.Hull}\left(
			\bR_{\geq 0}\omega_N,
			\ov{\cT}_N\cap \omega_N^{\perp}\right).
		\]
	\end{enumerate}
\end{proposition}
\begin{proof}
	The embedding $\phi$ from \autoref{sssec:preliminaries_for_induction} is an isometry of $V_{N-1}'$ onto $\omega_N^{\perp}$ and takes the fundamental domain $\ov{\cA}'$ to the intersection $\ov{\cA}\cap \omega_N^\perp$.
	From $W_{[N-1]}$-equivariance of the embedding it therefore follows that $\phi\left(\ov{\cT}'_{N-1}\right)$ is contained inside $\ov{\cT}_N\cap \omega_N^\perp$.
	Furthermore both sets are closed properly convex and $W_{[N-1]}$-invariant, and by maximality of $\ov{\cT}'_{N-1}$ in $V_{N-1}'$ it follows that they must coincide.

	To establish the second claim, note again that both $\omega_N$ and $\omega_N^\perp\cap \ov{\cT}_N$ belong to $\ov{\cT}_N\cap \omega_N^{\leq 0}$, so we have an inclusion one way.
	Suppose there is some point $z$ outside the specified convex hull.
	Project it onto $\omega_N^\perp$ along $\omega_N$, using the orthogonal decomposition $V_N=\omega_N\oplus \omega_N^{\perp}$, to obtain another point $z'$.
	Since the same projection of $\ov{\cT}_N\cap \omega_N^{\leq 0}$ would be another closed $W_{[N-1]}$-invariant subset, we conclude that $z'\in \ov{\cT}_N\cap \omega_N^{\perp}$.
	So we have that $\omega_N$ belongs to the interior of the segment $[z,z']$, itself contained in $\ov{\cT}_N$.
	Acting again by $W_{[N-1]}$ on the segment $[z,z']$, using that $\omega_N\in [z,z']$ is fixed, and that the $W_{[N-1]}$-orbit of $z'$ yields a basis of $\omega_N^\perp$, we conclude that $\omega_N$ is in the interior of $\ov{\cT}_N$.
	This, however, is absurd since $W_{[N]}$ acts properly discontinuously on the interior of $\ov{\cT}_N$ (\cite[Thm.~2]{Vinberg1971_Discrete-linear-groups-that-are-generated-by-reflections.} or \cite[Fact~3.8]{DancigerGueritaud_Convex-cocompactness-for-Coxeter-groups2023}), but $\omega_N$ has the infinite stabilizer $W_{[N-1]}$.
\end{proof}

\begin{remark}
Let us note that since $\omega_N$ has the infinite stabilizer $W_{[N-1]}$, the assertion in \cite[Remark 2.12]{oguisocantat} cannot hold as stated.
One must require that the sequence of elements in $W_{[N]}$ diverge in the coset space $W_{[N]}/W_{[N-1]}$, and to justify the ``easy induction'' claimed in the subsequent \cite[Remark 2.13]{oguisocantat} we have to appeal to a more delicate study of Coxeter groups.
\end{remark}

\subsubsection{Proof of \autoref{thm:divergent_recurrent_decomposition_of_limit_points}}
	\label{sssec:proof_of_thm:divergent_recurrent_decomposition_of_limit_points}
Consider the union of all the $W$-translates of all the hyperplanes $\omega_i^{\perp}$, denoted $H_{\bullet}:=\cup_{i\in [N]}W\cdot \omega_i^{\perp}$.
On the interior of $\cT$, this is a locally finite collection of hyperplanes (i.e. every point of the interior has a neighborhood intersecting only finitely many of the hyperplanes).
Note also that each hyperplane $\omega_i^{\perp}$ yields two half-spaces, one of which contains $\ov{\cC\cH}$.
Furthermore the intersection of all these half-spaces is precisely equal to $\ov{\cC\cH}$, since this cone is the dual of $\ov{\cT}$ and $\ov{\cT}$ is the closure of the span of all the $W$-orbits of all the $\omega_i$.

Let now $a\in \cA$ and $p\in \partial\ov{\cT}$ be two points and consider the segment $[a,p]\subset \ov{\cT}$.
We will parametrize points by
\[
	a_s:= (1-s)a + s p \quad \text{ for }s\in [0,1].
\]
Note also that if $[a,p]$ stays entirely in $\ov{\cC\cH}$, then $p$ necessarily belongs to the limit set of $W$.
We have three mutually exclusive possibilities:
\begin{enumerate}
	\item Either $[a,p]$ never intersects $H_{\bullet}$,
	\item or $[a,p]\cap H_{\bullet}=\{p\}$,
	\item or $\exists s\in [0,1)$ such that $a_s\in H_{\bullet}$.
\end{enumerate}
We will treat each separately.
Note that uniqueness of the representation in the \autoref{thm:divergent_recurrent_decomposition_of_limit_points} follows from the analysis in each case.

\subsubsection{Case (i)}
	\label{sssec:case_i}
In the first case, $p$ necessarily belongs to the limit set $\wtilde{\Lambda}$ and is not parabolic, so $p=p_r$.
Furthermore $S_r=[N]$ (and hence $S_d=\emptyset$) since $p$ cannot lie in a limit set of a (conjugate of a) smaller $W_S$, since those limit sets are contained in the $W$-orbit of some $\omega_i^\perp$.

\subsubsection{Case (ii)}
	\label{sssec:case_ii}
In the second case, suppose for simplicity that the index of the hyperplane to which $p$ belongs is $N$, i.e. that $p\in w_0\cdot \omega_N^\perp$ for some $w_0\in W$.
In the notation of \autoref{sssec:preliminaries_for_induction}, we have by induction that
\[
	w^{-1}_0p = \phi\Big( w_1 \cdot \left(p_d' + p_r'\right)\Big) \quad w_1 \in W_{[N-1]}', p_\bullet' \in \partial \ov{\cT}_{N-1}'
\]
and furthermore $p_d'\in \cA_{S_d}'$ with $S_{d}'\subseteq [N-1]$ and $p_r'\in \wtilde{\Lambda}_{S_r}'$ with $S_r'\subset [N-1]\setminus S_d'$.
Recall that $\phi(\omega_i')=\omega_i + \lambda \omega_N$ so we can write explicitly:
\begin{align*}
	p_d' & = \sum_{i\in S_d'} x_i \omega_i' && \text{ and so}\\
	p & = \underbrace{w_0 w_1}_{w}\left(
	\overbrace{\sum_{i\in S_d'}x_i \omega_i +
	\lambda\left(\sum_i x_i\right)\omega_N}^{q_d}
	+ \underbrace{\phi(p_r')}_{q_r}
	\right)
\end{align*}
where we use the notation of \autoref{thm:divergent_recurrent_decomposition_of_limit_points}.
The claim follows with $S_d:=S_d'\cup\{N\}$ and $S_r=S_r'$.

\subsubsection{Case (iii)}
	\label{sssec:case_iii}
Suppose there exists $a_s\in H_{\bullet}$ with $s\in (0,1)$ (with our choice, $a\notin H_{\bullet}$).
For simplicity of notation, suppose that $a_s\in w_0\cdot \omega_N^\perp$ for some $w_0\in W$, and we take the smallest value of $s$ possible; since $H_\bullet$ forms a locally finite family in the interior, the possible choices of $s\in (0,1)$ is discrete and not accumulating at $0$.

From \autoref{prop:codimension_1_decomposition} it follows that
\[
	p = w_0 \left(x\omega_N + y p_1\right)
	\text{ with } p_1 \in\partial\ov{\cT}_{N-1}', x>0, y\geq 0.
\]
Note that $x>0$ since $p\notin w\cdot \omega_N^\perp$.
If $y=0$ then $p$ is proportional to $w\cdot \omega_N$, which is its divergent part and there is no recurrent part.
Otherwise, by induction applied to $p_1$ we have that
\[
	p_1 = \phi\left(w_1 \left(p_d' + p_r'\right)\right) \quad
	w_1\in W_{[N-1]}', p_{\bullet}'\in \partial \ov{\cT}_{N-1}.
\]
We proceed analogously to case (ii) and write
\begin{align*}
	p_d' & = \sum_{i\in S_d'} x_i \omega_i' && \text{ and so}\\
	p & = \underbrace{w_0 w_1}_{w}\left(
	\overbrace{\sum_{i\in S_d'}yx_i \omega_i +
	\left[x+y\lambda\left(\sum_i x_i\right)\right]\omega_N}^{q_d}
	+ \underbrace{\phi(p_r')}_{q_r}
	\right)
\end{align*}
Note that we were able to factor out $w_1$ since it fixes $\omega_N$.
Again the claim follows, with $S_d=S_d'\cup\{N\}$ and $S_r=S_r'$.
\hfill \qed

\subsubsection{Additional notation for the decomposition}
	\label{sssec:additional_notation_for_the_decomposition}
For future use, let us introduce the map
\[
	\phi_{k,N}\colon V_k'\to V_N
\]
where $V_k'$ is spanned by $\omega_0',\dots,\omega_k'$ with quadratic form as in \autoref{eqn:quadraticform_omega_basis}, and with
\[
	\phi_{k,N}\left(\omega_i'\right) :=
	\omega_i + \tfrac{1}{k-1}\left(\omega_{k+1}+\dots + \omega_N\right)
\]
Note that again the embedding is conformal, up to a scaling factor of $(N-1)/(k-1)$ and the image is the orthogonal complement of $\omega_{k+1},\dots,\omega_N$.
We also have the corresponding map of groups $W'_{[k]}\into W_{[N]}$ such that $\phi_{k,N}$ is equivariant for the two group actions.

%%					end of subsec: Inductive structure of boundary
%%=============================================================================

%%%%%%%%%%%%%%%%%%%%%%%%%%%%%%%%%%%%%%%%%%%%%%%%%%%%%%%%%%%%%%%%%%%%%%%%%%%%%%%
%%% 				End of Section: Notation and the basic objects
%%%%%%%%%%%%%%%%%%%%%%%%%%%%%%%%%%%%%%%%%%%%%%%%%%%%%%%%%%%%%%%%%%%%%%%%%%%%%%%

%%%%%%%%%%%%%%%%%%%%%%%%%%%%%%%%%%%%%%%%%%%%%%%%%%%%%%%%%%%%%%%%%%%%%%%%%%%%%%%
%%% 				Start of Section: Volume function(s)
%%%%%%%%%%%%%%%%%%%%%%%%%%%%%%%%%%%%%%%%%%%%%%%%%%%%%%%%%%%%%%%%%%%%%%%%%%%%%%%

\section{Volume function(s)}
	\label{sec:volume_functions}

\paragraph{Outline of section}
In this section we begin our study of the volume function on the big cone of a Wehler $N$-fold, and state the main technical result \autoref{thm:volume_of_points} which describes the asymptotic behavior of $\vol(D+sA)$ when $s$ goes to zero in terms of the divergent-recurrent decomposition of $D$. This theorem immediately implies the first half of \autoref{main}.

%%=============================================================================
%%					start of subsec: A family of piecewise-homogeneous functions

\subsection{A family of piecewise-homogeneous functions}
	\label{ssec:a_family_of_piecewise_homogeneous_functions}

\subsubsection{Setup}
	\label{sssec:setup_a_family_of_piecewise_homogeneous_functions}
We keep the notation for cones, vectors, and groups from \autoref{sec:notation_and_the_basic_objects}.
For inequalities, the expression $A\leqapprox B$ will mean that there exists a constant $c>0$, possibly depending on $N$, such that $A\leq cB$.
Similarly $A\asymp B$ will be a shorthand for $A\leqapprox B$ and $B\leqapprox A$.
If the constants involved depend on further parameters, they will be indicated with a subscript, e.g. $A\asymp_X B$ means that there exists $c\geq 1$, depending on $X$, such that $\tfrac 1c A \leq B \leq cA$.

\begin{definition}[$k$-volume function]
	\label{def:k_volume_function}
	For $k\in \{0,1,\dots,N\}$ and for $a\in \ov{\cA}$ of the form $v=\sum_i x_i \omega_i$ define the function
	\[
		\vol_k(a):=\sum_{\substack{I\subseteq [N]\\
		|I|=k}} x_{i_1}\cdots x_{i_k}
	\]
	where by convention we set $\vol_0(a):=1$.

	Extend the function $\vol_k$ to $\cT$ in a $W$-equivariant way, so that $\vol_k(w\cdot a)=\vol_k(a)$ for any $w\in W,a\in \ov{\cA}$.
\end{definition}
It is immediate from the definitions that
\[
	\vol_k(\lambda a) = \lambda^k \vol_k(a) \quad \forall \lambda\in \bR_{>0}, \forall a\in \ov{\cT}.
\]
Our main interest will be in $\vol_{N}$ but to study it inductively, all the functions will have to be taken into account.
Note that $\vol_{N+1}$ will not appear in any considerations below; note also that $\vol_{N+1}$ vanishes on the codimension $1$ faces of $\cF$.

\subsubsection{Volume on the boundary of $\ov{\cA}$}
	\label{sssec:volume_on_the_boundary_of_A}
For future use, we explicitly evaluate the $k$-volume of points on the boundary of $\ov{\cA}$.
Since we will be interested in the coarse behavior only, we note that if $p\in \cA_S$ then
\[
	\vol_{k}(p) =
	\begin{cases}
		0, &\text{ if }k>|S|\\
		c_{k,p} & \text{ if }k\leq |S|
	\end{cases}
\]
where $c_{k,p}$ denotes a strictly positive constant, that does depend on $k$ and $p$.

%%					end of subsec: A family of piecewise-homogeneous functions
%%=============================================================================

%%=============================================================================
%%					start of subsec: Volumes and divergent-recurrent decomposition

\subsection{Volumes and divergent-recurrent decomposition}
	\label{ssec:volumes_and_divergent_recurrent_decomposition}

\subsubsection{Setup}
	\label{sssec:setup_volumes_and_divergent_recurrent_decomposition}
Fix now a point $p\in \partial\ov{\cT}$.
By \autoref{thm:divergent_recurrent_decomposition_of_limit_points} we can write $p=p_d+p_r$ and there exists $w\in W,q\in \partial\ov{\cT}$ such that $p=w\cdot q$, $p_d=w\cdot q_d,p_r=w\cdot q_r$ and furthermore $q_d\in \cA_{S_d}$ and $q_r\in \wtilde{\Lambda}_{S_r}$ with $S_d,S_r$ disjoint and $S_d$ of maximal cardinality among all choices of $w$.

Without loss of generality and to simplify notation, since the volume function is $\PsAut(X)$-invariant, we can assume that $w=\id$ and so $p=q,p_{\bullet}=q_{\bullet}$ for $\bullet\in\{d,r\}$.

\subsubsection{Decomposing the path}
	\label{sssec:decomposing_the_path}
For later use, to match the conventions from algebraic geometry, we will be interested in the points
\[
	a_s := s\cdot a + p, \quad s\in[0,1],\quad a\in\cA,
\]
which are not quite on the segment $[a,p]$, but are proportional to points on it and for most purposes give equivalent asymptotic quantities later on.

\subsubsection{Lim-sup and lim-inf}
	\label{sssec:lim_sup_and_lim_inf}
To express the possible behaviors, let us introduce the following notation:
\begin{equation}\label{supinf}\begin{split}
	\delta^{\sup}(p)&:=
	\limsup_{s\downarrow 0} \frac{\log \vol_N(a_s)}{\log s}\in\mathbb{R},\\
	\delta^{\inf}(p)&:=
	\liminf_{s\downarrow 0} \frac{\log \vol_N(a_s)}{\log s}\in\mathbb{R},
\end{split}\end{equation}
which are easily seen to be independent of the choice of $a\in{\cA}$, see e.g. \cite[Lemma 2.2 (1)]{jiangwang}.
If the $\limsup$ and $\liminf$ agree, we will denote the limit by $\delta(p)$.

Our main result regarding these quantities states:

\begin{theorem}[Volume of points]
	\label{thm:volume_of_points}
	Suppose $p=p_d+p_r$ is the decomposition of a point $p\in \partial \ov{\cT}$ as in \autoref{thm:divergent_recurrent_decomposition_of_limit_points}, with $S_d$ denoting the set of nonzero coefficients of $p_d$ in the basis $\{\omega_i\}$ (up to the $W$-action, and possibly empty).

	\begin{enumerate}
		\item \label{item:pure_divergence}
		If $p_r=0$ then $\delta(p)=N-|S_d|$.

		\item \label{item:pure_recurrence}
		If $p_d=0$ then
		\[
			\delta^{\sup}(p) = \tfrac12N \quad \text{ and }
			\delta^{\inf}(p)\in \left[1,\tfrac 12 N\right]
		\]
		Furthermore, for every $\delta\in \left[1,\tfrac 12 N\right]$ there exists $p_r\in \partial\ov{\cT}$ with $\delta^{\inf}(p_r)=\delta$.

		\item \label{item:mixed_recurrence_divergence}
		If both $p_d,p_r$ are nonzero, set $k:=N-|S_d|$ and recall we defined $\phi_{k,N}\colon V_k\to V_N$ in \autoref{sssec:additional_notation_for_the_decomposition}.
		Set $p_r'\in V_{k}$ such that $\phi_{k,N}(p_r')=p_r$ and note that $p_r'$ is itself a recurrent point in the boundary of the corresponding Tits cone.
		Let also $a'\in V_k'$ be such that $a,a'$ share the first $k+1$ coordinates in the $\omega_i,\omega_i'$ bases.
		Then
		\[
			\vol_N(s\cdot a + p ) \asymp \vol_k\left(s\cdot a'+p_r'\right)\text{ for }0<s\ll 1
		\]
		and hence $\delta^{\bullet}(p)=\delta^{\bullet}(p_r')$ for $\bullet\in \{\sup,\inf,\emptyset\}$.
		In particular
		\[
		\delta^{\sup}(p)=\tfrac 12 \left(N-|S_d|\right)\text{ and }\delta^{\inf}(p)\in\left[1,\tfrac 12 \left(N-|S_d|\right)\right]
		\]
		and any allowed value of $\delta^{\inf}$ is realized by some $p$.
	\end{enumerate}
\end{theorem}
\noindent
Note that if both $p_d,p_r$ are nonzero, then since $|S_d|+|S_r|\leq N+1$ and $|S_r|\geq 3$, we have that $N-|S_d|\geq 2$.

After some preliminary constructions in hyperbolic geometry, we will prove the first two parts of the theorem in \autoref{ssec:proof_of_the_theorem}.
The last part we address now.

\subsubsection{Preliminary notation}
	\label{sssec:preliminary_notation}
We will work in the basis $\omega_i,\omega_i'$ as in \autoref{sssec:additional_notation_for_the_decomposition} and abbreviate points as
\[
	\sum_{i=0}^N x_i \omega_i =: (x_0:\cdots :x_N)
\]
We may suppose that after an appropriate action of the group $W$ (and renumbering the basis elements) we have that $S_d=\{k+1,\dots,N\}$, that $S_r\subset [k]$, and the vectors are:
\begin{align*}
	p_d & = (0:\dots : 0: \mu_{k+1}:\dots : \mu_N)\\
	a_d & = (0:\dots:0:\alpha_{k+1}:\dots:\alpha_N)\\
	a_r & = (\alpha_0:\dots:\alpha_k:0\dots:0)\\
	a & = a_r + a_d\\
	p_r & = \left(\lambda_0:\dots :\lambda_k
	:\tfrac{\sum_i\lambda_i}{k-1}:\dots : \tfrac{\sum_i \lambda_i}{k-1}\right)\\
	p_r' & = \left(\lambda_0:\dots :\lambda_k\right)
\end{align*}
Let us now decompose the expression of interest as:
\[
	s\cdot a + p = (s\cdot a_r + p_r) + (s\cdot a_d + p_d).
\]
Note that for $0<s\ll 1$ we have that $s\cdot a_d + p_d\in \cA_{S_d}$ and its entries are uniformly bounded above and away from zero.
To compute $\vol_N(s\cdot a + p)$, we must apply an element $w(s)\in W_{S_r}$, which in particular fixes $s\cdot a_d+p_d$, such that $w(s)\cdot (s\cdot a_r+p_r)$ belongs to the fundamental domain $\ov{\cA}$, i.e. the first $k+1$ coordinates are nonnegative.

\autoref{thm:volume_of_points} (iii) then follows from the following estimate, together with the claim that $\vol_1(a_s')\leqapprox 1$ established in \autoref{cor:boundedness_of_volumes_on_F1} below.

\begin{proposition}[Reduction of volume]
	\label{prop:reduction_of_volume}
	Suppose that $a'=(x_0:\dots:x_k)$ with $x_i\geq 0$ and that $\mu_{k+1},\dots,\mu_N$ satisfy $\mu_i\in [\tfrac 1C,C]$ for some $C>0$.
	Set
	\[
		a:= \left(x_0:\dots:x_k:
		\tfrac{\sum_i x_i}{k-1}+\mu_{k+1}:\dots :
		\tfrac{\sum_i x_i}{k-1}+\mu_{N}\right)
	\]
	Then
	\[
		\vol_N(a)\asymp_C \vol_k\left(a'\right) \cdot \left[
		1+\vol_1(a')\right]^{N-k}
	\]
\end{proposition}
\begin{proof}
	Observe that $\vol_1(a')=\sum_i x_i$ by definition.
	Expanding now $\vol_N(a)$ using the definition, recall that it has $N+1$ terms, labeled by which one of the $N+1$ coordinates was omitted from the product.
	We have (denoting by $\widehat{x}$ a term that's omitted from the product):
	\begin{align*}
		\vol_N(a) & \asymp_C
		\left(\sum_{i=0}^k x_0\cdots \widehat{x_i}\cdots x_{k}\right)
		\left[1 + \vol_1(a')\right]^{N-k}
		+
		\left(x_0\cdots x_k\right) \left[1 +\vol_1(a')\right]^{N-k - 1}\\
		& =
		\vol_k(a')
		\left[1+\vol_1(a')\right]^{N-k} +
		\vol_{k+1}(a')[1+\vol_1(a')]^{N-k-1}
	\end{align*}
	Now it is elementary that for $x_i\geq 0$ we have that
	\[
		\vol_{k}(a')\vol_1(a')\geqapprox \vol_{k+1}(a')
	\]
	so the first term in the last expression dominates:
	\[
			\vol_N(a) \asymp_C \vol_k(a')\left[1 + \vol_1(a')\right]^{N-k}.
	\]
	giving the claim.
\end{proof}

%%					end of subsec: Volumes and divergent-recurrent decomposition
%%=============================================================================

%%%%%%%%%%%%%%%%%%%%%%%%%%%%%%%%%%%%%%%%%%%%%%%%%%%%%%%%%%%%%%%%%%%%%%%%%%%%%%%
%%% 				End of Section: Volume function(s)
%%%%%%%%%%%%%%%%%%%%%%%%%%%%%%%%%%%%%%%%%%%%%%%%%%%%%%%%%%%%%%%%%%%%%%%%%%%%%%%

%%%%%%%%%%%%%%%%%%%%%%%%%%%%%%%%%%%%%%%%%%%%%%%%%%%%%%%%%%%%%%%%%%%%%%%%%%%%%%%
%%% 				Start of Section: Hyperbolic geometry
%%%%%%%%%%%%%%%%%%%%%%%%%%%%%%%%%%%%%%%%%%%%%%%%%%%%%%%%%%%%%%%%%%%%%%%%%%%%%%%

\section{Hyperbolic geometry}
	\label{sec:hyperbolic_geometry}

\paragraph{Outline of section}
In this section we give the proof of \autoref{thm:volume_of_points}, by comparing the volume function to the hyperbolic distance from a basepoint, and estimating this distance along cusp excursions. The crucial construction of boundary points with prescribed volume behavior is achieved by gluing together pieces of hyperbolic geodesics.

%%=============================================================================
%%					start of subsec: Volumes along geodesics

\subsection{Volumes along geodesics}
	\label{ssec:volumes_along_geodesics}

\subsubsection{Setup}
	\label{sssec:setup_volumes_along_geodesics}
We keep the notation from \autoref{sec:notation_and_the_basic_objects}.
Consider a boundary point $p\in \partial \ov{\cT}$ such that $p\in \wtilde{\Lambda}_S$ for some $S\subseteq [N]$, and such that $p\notin W\cdot \wtilde{\Lambda}_{S'}$ for a smaller $S'$.
Suppose also that $p$ is not parabolic, i.e. $|S|\geq 3$.
In other words, the divergent-recurrent decomposition of $p$ only has a recurrent part and it is $p$ itself.

\subsubsection{Reparametrizations}
	\label{sssec:reparametrizations}
Consider as in \autoref{sssec:decomposing_the_path} the segment
\[
	a_s := s\cdot a + p \quad s\in [0,1]
\]
for some fixed $a\in \cF$.
We assume, for convenience and without loss of generality, that $\ip{a,p}=1/2$ and $\ip{a,a}=1$.
Note that $\ip{a_s,a_s}=s(1+s)$.
Introduce also $p_-:=a-p$ and note that $\ip{p_-,p_-}=0$.
Let us introduce a new coordinate $t$ defined implicitly by the identity:
\begin{align*}
	s^{-1/2} \frac{a_s}{(1+s)^{1/2}} & = (1+s)^{-1/2}\left[s^{1/2}a + s^{-1/2}p\right] \\
	& = e^{-t} p_- + e^{t}p
	 =: \eta_t
\end{align*}
Working out explicitly the value of $t$ we find that
\[
	s = e^{-2t}\frac{1}{1-e^{-2t}} \asymp e^{-2t} \text{ as }t\to +\infty
\]
Note that with our normalizations we have that $\ip{\eta_t,\eta_t}=1$, and therefore setting $\wtilde{a}_s:=a_s/\norm{a_s}$ we find that
\[
	\wtilde{a}_s = \eta_t \text{ with $s$ and $t$ as above.}
\]
Using the scaling properties analysis of the volume function, we will reduce the analysis of $\vol_N(a_s)$ to that of $\vol_N(\eta_t)$.

%%					end of subsec: Volumes along geodesics
%%=============================================================================

%%=============================================================================
%%					start of subsec: Calculations in the cusp

\subsection{Calculations in the cusp}
	\label{ssec:calculations_in_the_cusp}

\subsubsection{Heights and horoballs relative to a cusp}
	\label{sssec:heights_and_horoballs_relative_to_a_cusp}
Recall from \autoref{prop:convex_core_fundamental_domain_properties} that
\[
	\omega_{\widehat{ij}} = \left(\sum_{i=0}^N\omega_i\right) - (\omega_i + \omega_j)
\]
is an isotropic vector.
For a vector $v$ we set
\[
	{\rm Ht}_{ij}(v):= \frac{\norm{v}}{\ip{v,\omegaij}}
\]
whenever the quantity is defined and will refer to ${\rm Ht}_{ij}(v)$ as the \emph{height} of $v$ relative to the cusp corresponding to the parabolic point $\omegaij$.
Note that ${\rm Ht}_{ij}(\lambda v)={\rm Ht}_{ij}(v)$ for any scalar $\lambda>0$.

We also define the associated horoball:
\[
	{\rm Hor}_{ij}(L):=\left\lbrace v\in \cH\colon
	{\rm Ht}_{ij}(v) \geq L
	\right\rbrace.
\]
\subsubsection{Normalizations}
	\label{sssec:normalizations_horoball}
For the fundamental domain $\cF:=\ov{\cF}\cap \cH$ from \autoref{prop:convex_core_fundamental_domain_properties} we set
\[
	\cF_{ij}(L):= \cF \cap {\rm Hor}_{ij}(L)
	\quad \text{ and }\quad
	\cF^1_{ij}(L):= \cF_{ij}(L) \cap \cH^1.
\]
We will also want to normalize vectors to belong to a specified hyperplane.
Namely, given $a_0\in \cH$ and $K>0$, we set
\[
	\cF^{a_0,K}:= \cF \cap \{v\colon \ip{v,a_0}=K\}
	\quad
	\text{ and }
	\quad
	\cF^{a_0,K}_{ij}(L):= \cF_{ij}(L)\cap \cF^{a_0,K}.
\]

\begin{theorem}[Asymptotics of volume in the cusp, on hyperboloid]
	\label{thm:asymptotics_of_volume_in_the_cusp_on_hyperboloid}
	There exists $L_N>0$ such that if $L\geq L_N$, we have for every $a\in \cF^1_{ij}(L)$ that
	\begin{align*}
		\vol_k(a)\asymp_{L}
		\begin{cases}
			{\rm Ht}_{ij}(a) &\text{ if }k=1;\\
			{\rm Ht}_{ij}(a)^{N-2} &\text{ if }k = N.\\
		\end{cases}
	\end{align*}
\end{theorem}
\noindent Using the scaling properties of the functions involved, we will deduce in \autoref{sssec:proof_of_thm:asymptotics_of_volume_in_the_cusp_on_hyperboloid_assuming_thm:asymptotics_of_volume_in_the_cusp_on_hyperplane} the above claim from to the following equivalent form:
\begin{theorem}[Asymptotics of volume in the cusp, on hyperplane]
	\label{thm:asymptotics_of_volume_in_the_cusp_on_hyperplane}
	There exists $L_N>0$ such that if $L\geq L_N$, and given $a_0\in \cH$ and $K>0$, for any $a\in \cF^{a_0,K}_{ij}(L)$ we have
	\[
		\norm{a} \asymp_{L,a_0,K} {\rm Ht}_{ij}^{-1}(a)
		\text{ and }
		\vol_k(a) \asymp_{L,a_0,K}
		\begin{cases}
			1 &\text{ if }k=1;\\
			\norm{a}^{2} &\text{ if }k=N.\\
		\end{cases}
	\]
\end{theorem}

\subsubsection{Proof of \autoref{thm:asymptotics_of_volume_in_the_cusp_on_hyperboloid} assuming \autoref{thm:asymptotics_of_volume_in_the_cusp_on_hyperplane}}
	\label{sssec:proof_of_thm:asymptotics_of_volume_in_the_cusp_on_hyperboloid_assuming_thm:asymptotics_of_volume_in_the_cusp_on_hyperplane}
Fix once and for all $a_0,K$ (so we will omit dependence of constants on these choices) and suppose that $a$ is normalized so that $\ip{a_0,a}=K$ and then set $\wtilde{a}:=a/\norm{a}$.
Then we have
\begin{align*}
	\vol_k(\wtilde{a}) & =
	\vol_k\left(\frac{a}{\norm{a}}\right)
	= \norm{a}^{-k}\vol_k(a) \asymp_L \norm{a}^{-k}
	\begin{cases}
		1, &\text{ if }k=1;\\
		\norm{a}^2, &\text{ if }k=N.
	\end{cases}
\end{align*}
\autoref{thm:asymptotics_of_volume_in_the_cusp_on_hyperplane} also gives us
\begin{align*}
	\norm{a}\asymp_L {\rm Ht}_{ij}(\wtilde{a})^{-1}
\end{align*}
since ${\rm Ht}_{ij}(a)={\rm Ht}_{ij}(\wtilde{a})$.
Combined with the other conclusion of \autoref{thm:asymptotics_of_volume_in_the_cusp_on_hyperplane} we deduce \autoref{thm:asymptotics_of_volume_in_the_cusp_on_hyperboloid}.
\hfill \qed

\subsubsection{The structure of the horoball}
	\label{sssec:the_structure_of_the_horoball}
To proceed, we need a more detailed description of the horoball.
From \autoref{sssec:the_structure_at_the_cusp}, we know that the ``link'' of $\cF$ at the ray spanned by $\omegaij$ is generated by the rays of the vectors $\omega_{\widehat{ik}},\omega_{\widehat{\ell,j}}$ over all $k,\ell\in [N]\setminus \{i,j\}$.
The condition that $\omegaij+\lambda \omega_{\widehat{ik}} $ belongs to the horoball $\cF_{ij}(L)$, for $\lambda\geq 0$, translates to:
\[
	L\ip{\omegaij,\omegaij + \lambda \omega_{\widehat{ik}}} \leq \norm{\omegaij+\lambda \omega_{\widehat{ik}}}
\]
which leads to $\lambda \leqapprox \frac1{L^2}$ for a uniform constant.

\subsubsection{Proof of \autoref{thm:asymptotics_of_volume_in_the_cusp_on_hyperplane}}
	\label{sssec:proof_of_thm:asymptotics_of_volume_in_the_cusp_on_hyperplane}
Instead of normalizing $a$ to $\ip{a_0,a}=K$, it will be more convenient to take $a$ of the form
\[
	a = \omegaij
	+ \sum_{k} \alpha_k \omega_{\widehat{ik}}
	+ \sum_{\ell} \beta_{\ell} \omega_{\widehat{lj}}
\]
with $\alpha_k, \beta_{\ell}\in [0, 1/L^2]$.
Indeed, once $L$ is sufficiently large, $\ip{a_0,a}$ stays uniformly bounded above and away from zero and rescaling it accordingly will only affect the constants, not the conclusions.
Note that the above representation of $a$ is not unique, but we will establish the estimates up to some combinatorial multiplicative factors that will not affect the choice of such a representation.

Next, once $L$ is sufficiently large we find that
\[
	\norm{a}^2 \asymp \sum_k \alpha_k + \sum_{\ell} \beta_{\ell}
\]
since the $\omega_{\widehat{pq}}$ are isotropic but have nonvanishing, strictly positive pairwise inner products, and furthermore the terms quadratic in $\alpha_k,\beta_{\ell}$ are dominated by the linear terms.

Similarly, $\ip{a,\omegaij}\asymp \sum_k \alpha_k +\sum_{\ell} \beta_{\ell}$ from which we conclude that
\[
	{\rm Ht}_{ij}(a)^{-1} = \frac{\ip{a,\omegaij}}{\norm{a}} \asymp \norm{a}
\]
which was the first claim to be proved.

For the volume function, let us note that for $k=1$ we have more generally a bound on all of $\cF^{a_0,K}$, not just in the horoball.
Indeed $\cF^{a_0,K}$ is compact, and the continuous nonnegative function $\vol_k$ does not vanish anywhere on it.
By compactness it is uniformly bounded above and away from zero.
It remains to analyze $\vol_N(a)$.

It is clear that $\vol_N(a)$ is a polynomial $f(\alpha_\bullet,\beta_\bullet)$, which is furthermore symmetric in each set of variables.
Let us note that the constant term vanishes since $\vol_N(\omega_{\widehat{ij}})=0$.
Let us also note that the linear term doesn't vanish, since if we set all but one of the variables to zero we find that
\[
	\vol_N(\omega_{\widehat{ij}} + \alpha_0\omega_{\widehat{i0}}) = \alpha_0 + O\left(\alpha_0^2\right).
\]
We conclude that for $L$ large enough we have
\[
	\vol_N(a) \asymp \sum_k \alpha_k + \sum_{\ell} \beta_{\ell}
\]
as desired.
\hfill \qed

\noindent
To put all the cusps together, we will use the following basic estimate:

\begin{proposition}[Height and distance to basepoint]
	\label{prop:height_and_distance_to_basepoint}
	There exists $L_N>0$ such that if $L\geq L_N$, given any $a_0\in \cF^1$ we have for any $a\in \cF^1_{ij}(L)$ the estimate:
	\begin{equation}\label{asintoto}
		e^{\dist_{\rm hyp}(a_0,a)}
		\asymp_{a_0}
		{\rm Ht}_{ij}(a)
	\end{equation}
	for an appropriately normalized hyperbolic distance $\dist_{\rm hyp}$ on $\cH^1$.
\end{proposition}
\begin{proof}
	The identity is checked in the \Poincare upper half-space model which identifies $\cH^1$ with $(x_1,\ldots,x_{N-1},y)\in \bR^{N}$ and $y>0$, such that the cusp $\omegaij$ goes to $\infty$.
	
	Equip first the \Poincare model with the Riemannian metric $\frac{dy^2 + \sum_i dx_i^2}{y^2}$.
	The cusp $\cF^{1}_{ij}(L)$ gets mapped to a set of the form $F:=D\times \{y\geq \ell\}$ with $D\subset \bR^{N-1}$ a compact set.
	In this model, the hyperbolic distance from $a_0$ to a point $x\in F$ satisfies $\dist_{{\rm hyp}}(a_0,x)= \log y + O_{a_0}(1)$.
	As we will see below, we also have that ${\rm Ht}_{ij}=c_1 y^{c_2}$ for some constants $c_1,c_2>0$, so if we rescale the hyperbolic metric we find that
	\[
		{\rm Ht}_{ij}(a)\asymp_{a_0}e^{\dist_{{\rm hyp}'}(a_0, a)} \quad \text{for }a\in \cF^{1}_{ij}(L)
	\]
	as desired.

	It remains to justify that ${\rm Ht}_{ij}=c_1 y^{c_2}$.
	To verify it, two approaches are possible.
	One would be to go through the transition maps between the different models of hyperbolic space and check the identification, which would be quite tedious and require references to several points in the literature; this turns out to give that $c_2=c_1=1$.

	To verify just the existence of $c_1,c_2$, one can however proceed using Lie theory.
	Namely, the stabilizer of an isotropic vector in $\bR^{1,N}$ is a unipotent subgroup of $\SO_{1,N}(\bR)$ and clearly both $y$ and ${\rm Ht}_{ij}$ are constant on its orbits, so the level sets of the two functions agree.
	We also use the $1$-parameter subgroup $g_t\in \SO_{1,N}(\bR)$ with entries $e^t,e^{-t}$ and $1$, and expanding by a factor of $e^t$ the isotropic vector $\hat{\omega}_{ij}$.
	It gives a hyperbolic geodesic through the basepoint $a_0$ going at constant speed to the point at infinity $\hat{\omega}_{ij}$.
	We have that $y(g_t \cdot a_0)/y(a_0)= e^{\lambda_1 \cdot t}$ and ${\rm Ht}_{ij}(g_t \cdot a_0)/{\rm Ht}_{ij}(a_0)=e^{t}$ so the desired claim follows.
\end{proof}

\begin{proposition}[Distance to basepoint controls volume]
	\label{prop:distance_to_basepoint_controls_volume}
	Fix a basepoint $a_0\in \cM^{cc}:=\cC\cH^1/W$ and let $\phi(a):=\dist_{\rm hyp}(a_0,a)$ be defined on $\cM^{cc}$ and lifted to the cover $\cC\cH^1$.

	Then for any $a\in \cC\cH^1$ and $k\in [N]$ we have that
	\[
		\vol_k(a) \asymp_{a_0} \begin{cases}
			e^{\phi(a)} &\text{ if }k=1;\\
			e^{(N-2)\cdot\phi(a)} & \text{ if }k=N.
		\end{cases}
	\]
\end{proposition}
\begin{proof}
	We work in the fundamental domain $\cF^1\subset \cC\cH^1$ for the $W$-action.
	For any $L>0$, $\cF^1$ decomposes into the cusp neighborhoods $\cF_{ij}^1(L)$ and their complement which is (pre)compact.
	On a compact set the exponential of any continuous function is comparable to that of any other, while in each $\cF_{ij}^1(L)$ we have that ${\rm Ht}_{ij}(\bullet)\asymp \exp\big(\dist_{\rm hyp}(a_0,\bullet)\big)$, something that again can be seen in the upper half-space model.
	The claim then follows from \autoref{thm:asymptotics_of_volume_in_the_cusp_on_hyperboloid}.
\end{proof}

\begin{corollary}[Boundedness of volumes on $\cF^1$]
	\label{cor:boundedness_of_volumes_on_F1}
	For any $k\leq N$ we have that $\vol_k$ restricted to $W\cdot \cF^1=\cC\cH^1$ is uniformly bounded away from zero.

	Furthermore, suppose $\omega$ is any geodesic ray starting at $a_0$ and staying in $\cC\cH^1$, parametrized isometrically as $\eta_t$ with $\eta_0=a_0$ and $t\geq 0$.
	Then we have that
	\[
		e^{t}\cdot e^{-\phi(\eta_t)}\geqapprox_{a_0} 1.
	\]
	In the parametrization where $a_s=sa+p$ and $s^{1/2}\asymp e^{-t}$ the last inequality is equivalent to
	\[
		\vol_1(a_s) \leqapprox 1.
	\]
\end{corollary}
\begin{proof}
	The first claim follows immediately from \autoref{prop:distance_to_basepoint_controls_volume}.
	The second is also geometrically clear, since it is equivalent to saying that $t-\phi(\eta_t)\geq -C_{a_0}$.
\end{proof}

%%					end of subsec: Calculations in the cusp
%%=============================================================================

%%=============================================================================
%%					start of subsec: Gluing geodesics

\subsection{Gluing geodesics}
	\label{ssec:gluing_geodesics}

We establish below several elementary estimates on geodesics in hyperbolic space, which will later allow us to construct recurrent boundary points with prescribed behavior.
It will be useful to alternate between the ball and \Poincare upper half-space model of hyperbolic space.

The properties are standard, but we include them for completeness.
In this section, we will work with explicit forms of the hyperbolic metric and since all our statements are invariant under one initial rescaling of the hyperbolic metric, the results apply to the normalization in \autoref{prop:height_and_distance_to_basepoint}.

\subsubsection{Standard cusp excursion}
	\label{sssec:standard_cusp_excursion}
Suppose $\bH^N=\{(x_1,\dots x_{N-1},y)\colon y>0\}$ is the upper half-space model of hyperbolic space, and $w\colon \bH^N\to \bH^N$ is an isometry fixing $\infty$ of the form $x \mapsto x + \tau_w$ with $\tau_w\in \bR^{N-1}\setminus \{0\}$.
Note that scalings and elements of $\Orthog_{N-1}(\bR)$ are also isometries of $\bH^N$, so by changing coordinates we can normalize $\tau_{w}$ to any other nonzero vector.

For $x_0\in \bH^N$, a \emph{standard cusp excursion} associated to $x_0$ and $w$ is the geodesic segment connecting $x_0$ and $w^k x_0$ for any $k\in \bZ$.
When $w$ is part of a larger discrete group of isometries $W$ of $\bH^N$, we will use the same terminology for the image in the quotient $\bH^N/W$.
For a geodesic segment $\gamma$, we will denote by $s_{\gamma},e_{\gamma}$ its starting and ending points respectively, and by $ts_{\gamma},te_{\gamma}$ its tangent vectors at those points.

For the next statement, recall from \autoref{sec:notation_and_the_basic_objects} that $\cH^1$ denoted the hyperboloid of unit norm vectors and the convex hull of the limit set is $\cC\cH$.
We will denote by $\cM^{cc}:=\cC\cH^1/W$ and $\cM:=\cH^1/W$ the associated quotients, so that $\cM$ is a complete hyperbolic manifold and $\cM^{cc}$ is its convex core.
When $N=2$ we have $\cC\cH=\cH$.
For a point $x\in \cM$ we will denote by $T_x^1\cM$ the unit norm tangent vectors at $x$.

\begin{proposition}[Can go into the cusp anytime]
	\label{prop:can_go_into_the_cusp_anytime}
	Suppose $N=2$.
	Given any $x_0\in \cM^{cc}$, there exist $L_0>0,\ve_0>0$ with the following properties.
	Given any nonempty open neighborhood $U\subset T^1_{x_0}\cM^{cc}$ of size less than $\ve_0$ and any $L\geq L_0$, there exists $v\in U$ and a standard cusp excursion $\gamma$, with the following properties:
	\begin{enumerate}
		\item It starts at $x_0$ in the direction $v$, i.e. $s_{\gamma}=x_0,ts_{\gamma}=v$.
		\item It has length $L+O_{x_0}(1)$.
		\label{item:length_plus_O1}
		\item The endpoint is also $x_0$, but the tangent vector is near $-v$, i.e. $e_{\gamma}=x_0$ and $\dist(-v,te_{\gamma})\leq \ve_0$.
	\end{enumerate}
\end{proposition}
\noindent For general $N\geq 2$ an analogous proposition holds, but the neighborhood $U$ would have to contain some points of the limit set as well.
\begin{proof}
	When $N=2$, the limit set is the entire circle at infinity, and furthermore parabolic points are dense in the limit set (apply \cite[12.1.3]{Ratcliffe2019_Foundations-of-hyperbolic-manifolds} to the closure of the set of parabolic points).
	So for any neighborhood $U$ of $T^1_{x_0}\cM^{cc}$, there exists a $v_0$ such that the geodesic starting at $x_0$ in the direction $v_0$ lands at a parabolic point.
	Thus, there exists an isometry $w_0\in W$ such that, switching to the upper half-space model, we have that $x_0=(0,y_0)\in \bH^2$, that $w_0(x)=x+1$, and the tangent vector $v_0$ points vertically.

	Recall that the hyperbolic geodesic connecting $x_0$ with $x_0+k$ is the arc of the circle of radius $r_k:=\sqrt{(k/2)^2 + y_0^2}$ and center $(k/2,0)$.
	An elementary integral (involving only the integration of rational functions) reveals that the length of the arc is
	\[
		\log\left(r_k + \tfrac{|k|}2\right) - \log \left(r_k - \tfrac{|k|}2\right) =
		\log |k| + O_{y_0}(1)
	\]
	where we assume that $k\neq 0$.
	If $|k|$ is large enough the geometric statements about the starting and ending vectors are directly verified (note that tangent spaces at $x_0$ and $x_0+k$ are identified by the isometry $x\mapsto x+k$).
	We can also choose $k$ so that $L=\log|k| + O_{x_0}(1)$, in fact $L=\log|k| + o_{x_0,L}(1)$.
\end{proof}

We will also need the following estimates for piecewise geodesic segments.

\begin{proposition}[Gluing geodesics]
	\label{prop:gluing_geodesics}
	There exist $L_1>0,\ve_1>0$ with the following properties.
	Let $\alpha,\beta$ be two geodesics in $\bH^N$ of lengths $\ell_\alpha,\ell_\beta>L_1$.
	Suppose that $\alpha$ ends at the starting point of $\beta$, i.e. $e_{\alpha}=s_{\beta}$ and the tangent vectors are near, i.e. $\dist(te_{\alpha},ts_{\beta})<\ve_1$.
	Then the geodesic $\gamma$ connecting $s_{\alpha}$ to $e_{\beta}$ satisfies:
	\begin{enumerate}
		\item The starting vector of $\gamma$ is exponentially close to that of $\alpha$, namely
		\[
			\dist(ts_{\alpha},ts_{\gamma})\leqapprox e^{-\ell_{\alpha}}.
		\]
		\item We have that
		\[
			\ell_{\alpha}+ \ell_{\beta} - O(1) \leq \ell_{\gamma} \leq \ell_{\alpha}+\ell_{\beta}.
		\]
		\item Furthermore, the Hausdorff distance between $\gamma$ and $\alpha\cup \beta$ is also $O(1)$.
	\end{enumerate}
\end{proposition}
\noindent Note that the estimate on the angle between $ts_{\alpha}$ and $ts_{\gamma}$ improves as $\alpha$ gets longer.
This, together with the uniformity of estimates, will allow us to show that certain geodesic rays exist and are approximated by prescribed piecewise geodesic segments.
\begin{proof}
	We work in the ball model, and since the two geodesics $\alpha,\beta$ are contained in a $2$-dimensional hyperbolic subspace, we can assume $N=2$.
	We choose coordinates so that $s_{\alpha}=(0,0)$ and $e_{\alpha}=(x,0)$ with $x\in (0,1)$.
	Again an elementary integral shows that $\ell_{\alpha}=d(s_{\alpha},e_{\alpha})=-\log(1-x)+O(1)$, or equivalently $1-x\asymp e^{-\ell_{\alpha}}$.

	The extreme case of all estimates is when $\beta$ is along the ray starting at $e_\alpha$ at angle $\theta$ with $|\theta|=\ve_1$.
	Such rays are contained in a cone based at $0$ of width $\asymp|\tan\theta|(1-x)\leq e^{-\ell_{\alpha}}$.
	The other claims also follow from elementary hyperbolic geometry and this constraint.
\end{proof}

We can now put the two previous propositions together:

\begin{proposition}[Arbitrary cusp excursions, $N=2$]
	\label{prop:arbitrary_cusp_excursions}
	Suppose $N=2$ and $x_0\in \cM^{cc}$, $L_0>0$ is as in \autoref{prop:can_go_into_the_cusp_anytime}.
	All $O(1)$ bounds in this proposition are assumed to depend on $x_0$.

	Suppose $\ell_i,i\geq 0$ is a sequence of positive reals with $\ell_i>L_0+O(1)$ where the $O(1)$ is again as in \autoref{prop:can_go_into_the_cusp_anytime}\ref{item:length_plus_O1}.
	Then, for any open set $U\subset T^{1}_{x_0}\cM$, there exists an infinite geodesic ray $\gamma$, starting at $x_0$ and with initial tangent vector in $U$, with the following properties:
	\begin{enumerate}
		\item The geodesic ray has a decomposition into geodesic segments
		\[
			\gamma = \coprod_{i\geq 0} [x_i,x_{i+1}) \quad x_i\in \gamma
		\]
		with lengths $\ell_{[x_i,x_{i+1})}=\ell_i+O(1)$ and $\dist(x_0,x_i)\leq O(1)$.
		\item Furthermore $[x_i,x_{i+1})$ is at bounded Hausdorff distance from a standard cusp excursion based at $x_0$, of length $\ell_i + O(1)$.
		\item Let $m_i\in(x_i,x_{i+1})$ denote the midpoint of the segment.
		Then the distance function $\dist_{\cM}(x_0,-)$ restricted to $(x_i,x_{i+1})$ is within $O(1)$ of the function that linearly grows from $0$ at $x_i$ to $\tfrac 12 \ell_i$ at $m_i$, and then linearly decreases back to $0$.
	\end{enumerate}
\end{proposition}
\begin{proof}
	It follows from \autoref{prop:can_go_into_the_cusp_anytime} that we can form a piecewise geodesic segment $\gamma=\coprod \gamma_i$ with $\gamma_i$ starting and ending at $x_0$, with most of the listed properties.
	We only need to verify the claim about the distance function to $x_0$ restricted to each $\gamma_i$.
	This can again be verified in the \Poincare model, as was done in the proof of \autoref{prop:can_go_into_the_cusp_anytime}.
	Indeed we need to minimize the distance from $(x,y)$ to the points $(k,y_0)$ over all $k\in \bZ$ (and $y\geq y_0$).
	We see that this distance is $\log(y/y_0)+O(1)=\log(y)+O(1)$ which yields the required claims.

	To deduce that a geodesic ray exists with the same properties, we apply inductively \autoref{prop:gluing_geodesics}.
	First we apply it to the segments $\gamma_0,\gamma_1$ to obtain a geodesic $\gamma_1'$.
	Then we apply it to $\gamma_1'$ and $\gamma_2$ to obtain $\gamma_2'$.
	So at each step we obtain a geodesic segment $\gamma_i'$ connecting $x_0$ to $x_0$ but staying near $\gamma_0\coprod\dots \coprod \gamma_i$, and then glue it to $\gamma_{i+1}$ to obtain $\gamma_{i+1}'$.
	Note that by the estimate on the angle from \autoref{prop:gluing_geodesics}, the angles of $\gamma_i'$ form a Cauchy sequence that converges exponentially fast to an angle, and hence to a geodesic ray $\gamma$ as required.
\end{proof}

\begin{theorem}[Arbitrary cusp excursions]
	\label{thm:arbitrary_cusp_excursions}
	Suppose $N\geq 2$, $x_0\in \cM^{cc}$, and $U\subset T^1_{x_0}\cM$ is an open set containing a $W$-orbit of the limit set $\Lambda_{[2]}$.

	Then there exists $L_1>0$, such that for any sequence $\ell_i,i\geq 0$ with $\ell_i>L_1 + O(1)$, a geodesic ray $\gamma$ satisfying the same conclusions as in \autoref{prop:can_go_into_the_cusp_anytime} exists, with $O(1)$ bounds depending on $x_0$.
\end{theorem}
\begin{proof}
	Fix $x_0'\in \cM^{cc}$ to be a point on the totally geodesic boundary of $\cM^{cc}$, corresponding to the subgroup $W_{[2]}$ of $W$.
	Using the identification $T_x^1\cM^{cc}\isom \partial \bH^N$ for $x=x_0,x_0'$, let $U'\subset T_{x_0'}^1 \cM^{cc}$ be the corresponding open set.
	Then \autoref{prop:arbitrary_cusp_excursions} applies to $x_0'$ and the open set $U'$ (intersected with the totally geodesic copy of $\bH^2$) to give a geodesic ray $\gamma'$ starting at $x_0'$ and endpoint $p\in \partial \bH^N$.
	It is now immediate that the corresponding geodesic $\gamma$ starting at $x_0$ and with the same endpoint has the required properties, up to additive $O(1)$-terms depending on $x_0,x_0'$.
\end{proof}

%%					end of subsec: Gluing geodesics
%%=============================================================================

%%=============================================================================
%%					start of subsec: Proof of the theorem

\subsection{Proof of the theorem}
	\label{ssec:proof_of_the_theorem}

Using the estimates established in the previous two sections, we can now proceed to the proof \autoref{thm:volume_of_points}.

\subsubsection{The case of no recurrence}
	\label{sssec:the_case_of_no_recurrence}
Suppose that we are in case \ref{item:pure_divergence}, i.e. $p_r=0$ in the decomposition $p=p_d+p_r$.
Then we compute directly:
\[
	\vol_N(s\cdot a + p_d) \asymp \vol_N(s\omega_0 + \dots
	+ s\omega_{N-|S_d|} + \omega_{N-|S_d|+1}+\cdots + \omega_N)
	\asymp s^{N-|S_d|}
\]
hence giving the claim in this case.

\subsubsection{The case of pure recurrence}
	\label{sssec:the_case_of_pure_recurrence}
Suppose now that we are in \ref{item:pure_recurrence}, i.e. $p_d=0$ in the decomposition.
Recall that $a_s=s\cdot a + p_r$, that $\norm{a_s}\asymp s^{1/2}$ and that we set $\wtilde{a}_s:=a_s/\norm{a_s}$ and a reparametrization $\wtilde{a}_s=\eta_t$ with $s^{1/2}\asymp e^{-t}$, as in \autoref{sssec:reparametrizations}.

Then we have that
\begin{equation}\label{qz}\begin{split}
	\vol_N(a_s) & = \norm{a_s}^N\vol_N\left(\frac{a_s}{\norm{a_s}}\right)
	\asymp s^{N/2}\vol_N\left(\wtilde{a}_s\right)\asymp
	\\
	& \asymp e^{-t\cdot N}\cdot \vol_N(\eta_t)
	\asymp e^{-t\cdot N}\cdot e^{(N-2)\phi(\eta_t)}
\end{split}\end{equation}
by \autoref{prop:distance_to_basepoint_controls_volume} that also defines $\phi$.

Now $\phi\geq 0$ and for a recurrent point we have that $\liminf \phi(\eta_t)\leq O(1)$ so $\liminf \phi(\eta_t)/t=0$.
We thus find:
\begin{align*}
	\limsup_{s\downarrow 0}\frac{\log \vol_N(a_s)}{\log s}
	& = \limsup_{t\to +\infty}\frac{-t\cdot N + (N-2)\phi(\eta_t)}{-2t}\\
	& = \tfrac{N}{2} - \tfrac{N-2}{2}\cdot \liminf_{t\to +\infty}\frac{\phi(\eta_t)}{t} = \tfrac{N}{2}
\end{align*}
which was one half of the claim.
With an analogous calculation we find
\[
	\liminf_{s\downarrow 0 }\frac{\log \vol_N(a_s)}{\log s}
	=
	\tfrac{N}{2} - \tfrac{N-2}{2}\limsup_{t\to +\infty}\frac{\phi(\eta_t)}{t}.
\]
We do have $\phi(\eta_t)-t\leq O(1)$ by \autoref{cor:boundedness_of_volumes_on_F1} so $\limsup \frac{\phi(\eta_t)}{t}\in [0,1]$.
To prove the remaining part of \autoref{thm:volume_of_points}\ref{item:pure_recurrence} it suffices to show that any value in $[0,1]$ can be achieved by an appropriate choice of $p_r$.

To do so, we invoke \autoref{thm:arbitrary_cusp_excursions} with parameters $\ell_i$ to be specified.
We obtain a geodesic ray $\gamma$ decomposed into subsegments $\gamma_i$ of length $\ell_i+O(1)$, such that on each $\gamma_i$ the function $\phi(t)$ grows linearly from $0$ to $\ell_i/2$, and then decreases linearly to $0$, all up to a uniform additive error.

It follows that the expression $\frac{\phi(\eta_t)}{t}$ achieves its local maxima at the midpoints of the $\gamma_i$ (up to $O(1)$ error) so
\[
	\limsup_{t\to +\infty} \frac{\phi(\eta_t)}{t}
	= \limsup_{n\to +\infty} \frac{\ell_n/2 + O(1)}{\left[\ell_1 + \cdots + \ell_{n-1}\right] + \ell_{n}/2+O(n)}
\]
Let us now take $\ell_n:= L^n$ with $L>1$ and calculate:
\[
	\frac{\frac{L^n}{2} + O(1)}{\frac{L^n-1}{L-1} + \frac{L^n}{2} + O(n)} =
	\frac{\frac{1}{2} + O\left(L^{-n}\right)}{\frac{1}{L-1}+\frac{1}{2}+ O(nL^{-n}) }
	\xrightarrow{n\to +\infty}\frac{L-1}{L+1}
\]
As $L$ ranges in $(1,+\infty)$ the above limit ranges in $(0,1)$.

To obtain the extreme value $0$, take the sequence $\ell_n:=n^{10}$ (or any other polynomial of degree at least $2$ will suffice).
To obtain the value $1$, take $\ell_n : = 10^{10^n}$ or any other sequence that grows sufficiently fast.

This completes the proof of \autoref{item:pure_recurrence} and the theorem. \hfill \qed

\subsubsection{Summary}\label{sommario}
For later usage, we summarize here what this construction gives us in the recurrent case (ii), where $\delta\in [1,N/2]$ is a given real number and $p=p_r$ is the point in $\partial\ov{\mathcal{T}}$ given by \autoref{thm:volume_of_points} with $\delta^{\rm inf}(p_r)=\delta, \delta^{\rm sup}(p_r)=\frac{N}{2}$.

Recalling \eqref{qz}, the inequalities $\phi(\eta_t)\geq 0$ and $\phi(\eta_t)\leq t+C$ (from \autoref{cor:boundedness_of_volumes_on_F1}), after going back to the $s$ variable show that for all $s\geq 0$ small we have
\begin{equation}\label{cuz1}
\vol_N(a_s)\geq C^{-1}s^{\frac{N}{2}},
\end{equation}
\begin{equation}\label{cuz2}
\vol_N(a_s)\leq Cs.
\end{equation}
Also, by the above construction, there is a sequence $t_n\to+\infty$ (where $\phi(\eta_t)$ achieves its local maxima) such that if $t'_n$ is any other sequence with $|t'_n-t_n|\leq C$ (for some fixed $C$), then we have the very crude bound
\begin{equation}\label{cuz3a}
\phi(\eta_{t'_n})\geq \frac{N-2\delta}{N-2}t'_n - o(t_n).
\end{equation}
Plugging this into \eqref{qz}, and letting as usual $s'_n=e^{-2t'_n}, s_n=e^{-2t_n},$ this gives
\begin{equation}\label{cuz3}
\vol_N(a_{s'_n})\geq C^{-1} (s'_n)^{\delta+o(1)},
\end{equation}
and so along the sequence $s'_n$ the quantity
\[
\frac{\log\vol_N(a_{s'_n})}{\log s'_n},
\]
converges to its liminf, which is $\delta^{\rm inf}(p_r)=\delta.$
Similarly, there is a sequence $\ti{t}_n\to+\infty$ with $\phi(\eta_{\ti{t}_n})=0$, such that
that if $\ti{t}'_n$ is any other sequence with $|\ti{t}'_n-\ti{t}_n|\leq C$ (for some fixed $C$), then we have
\begin{equation}\label{cuz4a}
\phi(\eta_{\ti{t}'_n})\leq C',
\end{equation}
for some fixed $C'$.
Plugging this into \eqref{qz}, and letting $\ti{s}'_n=e^{-2\ti{t}'_n}, \ti{s}_n=e^{-2\ti{t}_n},$ this gives
\begin{equation}\label{cuz4}
\vol_N(a_{\ti{s}'_n})\leq C(\ti{s}'_n)^{\frac{N}{2}},
\end{equation}
which along this sequence matches \eqref{cuz1} up to a fixed multiplicative factor, and also shows that
along the sequence $\ti{s}'_n$ the quantity
\[
\frac{\log\vol_N(a_{\ti{s}'_n})}{\log \ti{s}'_n},
\]
converges to its limsup, which is $\delta^{\rm sup}(p_r)=\frac{N}{2}.$

%%					end of subsec: Proof of the theorem
%%=============================================================================

%%%%%%%%%%%%%%%%%%%%%%%%%%%%%%%%%%%%%%%%%%%%%%%%%%%%%%%%%%%%%%%%%%%%%%%%%%%%%%%
%%% 				End of Section: Hyperbolic geometry
%%%%%%%%%%%%%%%%%%%%%%%%%%%%%%%%%%%%%%%%%%%%%%%%%%%%%%%%%%%%%%%%%%%%%%%%%%%%%%%

%%%%%%%%%%%%%%%%%%%%%%%%%%%%%%%%%%%%%%%%%%%%%%%%%%%%%%%%%%%%%%%%%%%%%%%%%%%%%%%
%%% 				Start of Section: Main example
%%%%%%%%%%%%%%%%%%%%%%%%%%%%%%%%%%%%%%%%%%%%%%%%%%%%%%%%%%%%%%%%%%%%%%%%%%%%%%%

\section{Regularity of the volume function}
	\label{sec:main_example}

\paragraph{Outline of section}
In this section we prove \autoref{coro}, which describes the sharp regularity of the volume function at the pseudoeffective boundary on Wehler $N$-folds.

\subsection{\texorpdfstring{$C^1$}{C1} differentiability at the endpoint}

Before we give the proof of \autoref{coro}, let us observe the following general fact, which seems to be well-known but whose proof is not available in the literature:
\begin{proposition}\label{diff}
Let $X^n$ be a smooth projective variety, $D$ a pseudoeffective $\mathbb{R}$-divisor and $A$ an ample $\mathbb{R}$-divisor.
Then $f:[0,1]\to\mathbb{R}_{\geq 0}$ given by
$$f(s)=\vol(D+sA),$$
is $C^1$ differentiable on $[0,1]$, and
\begin{equation}\label{derivata}
f'(s)=n\langle (D+sA)^{n-1}\rangle \cdot A,
\end{equation}
holds for all $s\in [0,1]$.
\end{proposition}
Here $\langle\cdot\rangle$ denotes the positive intersection product \cite{bfj}.

\begin{proof}
Thanks to the differentiability result for the volume \cite[Theorem A]{bfj}, \cite[Corollary C]{lazarsfeldmustata}, we know that $f$ is $C^1$ differentiable on $(0,1]$ and that \eqref{derivata} holds for $s\in (0,1]$.
To prove $C^1$ differentiability at the endpoint $s=0$, we first note that by definition \cite[Definition 2.10]{bfj} we have
\[\lim_{s\downarrow 0}\langle (D+sA)^{n-1}\rangle=\langle D^{n-1}\rangle,\]
so if we define \[B:=n\langle D^{n-1}\rangle\cdot A\in\mathbb{R}_{\geq 0},\]
then we have
\begin{equation}\label{lim2}
\lim_{s\downarrow 0}f'(s)=B,
\end{equation}
and from this it follows $f(s)$ is right differentiable at $s=0$ with derivative equal to $B$. Indeed,
given any $0<s$, since $f$ is continuous on $[0,s]$ and differentiable on $(0,s]$, the mean value theorem gives us $\xi_s\in (0,s)$ such that
\[\frac{f(s)-f(0)}{s}=f'(\xi_s),\]
and letting $s\downarrow 0$ and using \autoref{lim2} gives the result.
%and so we have to show that $f(s)$ is right differentiable at $s=0$ with derivative equal to $B$. Indeed,
%given any $0<t<s$, since $f$ is $C^1$ on $[t,s]$ we can write
%\[f(s)=f(t)+\int_t^sf'(u)du,\]
%and rearranging and letting $t\to 0$
%\[\frac{f(s)-f(0)}{s}=\limsup_{t\downarrow 0}\frac{1}{s}\int_t^sf'(u)du,\]
%and so also
%\[\frac{f(s)-f(0)}{s}-B=\limsup_{t\downarrow 0}\frac{1}{s}\int_t^s(f'(u)-B)du,\]
%and thanks to \eqref{lim2}, given $\ve>0$ we can find $s_0>0$ such that $|f'(u)-B|\leq \ve$ for $0<u<s_0$. Thus, for $s<s_0$ we have
%\[\left|\frac{f(s)-f(0)}{s}-B\right|\leq \limsup_{t\downarrow 0}\frac{1}{s}\int_t^s|f'(u)-B|du\leq \ve,\]
%as desired.
\end{proof}
\begin{remark}
The argument extends verbatim to the case when $D$ is replaced by a pseudoeffective $(1,1)$-class and $A$ by a K\"ahler class, by using the $C^1$ differentiability result of \cite{wittnystrom} in place of \cite{bfj,lazarsfeldmustata}. It would also immediately extend to arbitrary compact K\"ahler manifolds provided that the conjectured differentiability is proved in that setting.
\end{remark}

\subsection{Failure of higher regularity}
Let now $X$ be a Wehler $N$-fold, $N\geq 3$. Thanks to \autoref{thm:volume_of_points} we know that there exists a pseudoeffective $\mathbb{R}$-divisor $D$ such that
\begin{align}
	\limsup_{s\downarrow 0} \frac{\log \vol(D+sA)}{\log s}&=\frac{N}{2},\label{supinfa}\\
	\liminf_{s\downarrow 0} \frac{\log \vol(D+sA)}{\log s}&=1,\label{supinfb}
\end{align}
for any ample divisor $A$.
The following result, together with \autoref{diff}, proves \autoref{coro}:

\begin{proposition}
For such $X,D,A$, we have that the function $f:[0,1]\to\mathbb{R}_{\geq 0}$ given by
$$f(s)=\vol(D+sA),$$
is not $C^{1,\alpha}$ on $[0,\ve)$ for any $\ve>0$ and any $\alpha>0$.
\end{proposition}
\begin{proof}
We argue by contradiction, and suppose that $f$ is $C^{1,\alpha}$ on $[0,\ve)$ for some $\ve>0$ and $\alpha>0$.
Using the mean value theorem, we can write for any given $s\in [0,\ve)$
$$f(s)=sf'(\xi),$$
for some $0\leq \xi\leq s$. Now thanks to \eqref{derivata} we have
\[f'(0)=n\langle D^{n-1}\rangle\cdot A.\]
Next, we claim that
\[\langle D^{n-1}\rangle=0.\]
For this we use an argument from \cite[Theorem 6.2 (7)$\leq$(1)]{lehmann} (using the basic properties of the positive product \cite{bfj}): if we had $\langle D^{n-1}\rangle\neq 0$ then we would have $\langle D^{n-1}\rangle\cdot A>0$, and therefore
\[\langle (D+sA)^{n-1}\rangle\cdot A>0,\]
for all $s\geq 0$, hence
$$Cs\leq s\langle (D+sA)^{n-1}\rangle\cdot A=\langle (D+sA)^{n-1}\cdot sA\rangle\leq \langle (D+sA)^n\rangle=\vol(D+sA),$$
and this would imply that for $s>0$ small we have
$$\frac{\log\vol(D+sA)}{\log s}\leq 1+\frac{\log C}{\log s},$$
hence
$$\limsup_{s \downarrow 0} \frac{\log \vol(D+sA)}{\log s} \leq 1,$$
which contradicts \eqref{supinfa}.

We thus conclude that $f'(0)=0$, and since $f'$ is assumed to be H\"older continuous of order $\alpha>0$ on $[0,\ve)$, we have $|f'(\xi)|\leq C\xi^\alpha\leq Cs^\alpha$ for some fixed $C>0$, and so
$$f(s)\leq Cs^{1+\alpha},$$
for $s$ close to $0$.
Therefore
$$\frac{\log\vol(D+sA)}{\log s}\geq 1+\alpha+\frac{\log C}{\log s},$$
which implies
$$\liminf_{s \downarrow 0} \frac{\log \vol(D+sA)}{\log s} \geq 1+\alpha,$$
contradicting \eqref{supinfb}.
\end{proof}

%%%%%%%%%%%%%%%%%%%%%%%%%%%%%%%%%%%%%%%%%%%%%%%%%%%%%%%%%%%%%%%%%%%%%%%%%%%%%%%
%%% 				Start of Section: From volume to sections
%%%%%%%%%%%%%%%%%%%%%%%%%%%%%%%%%%%%%%%%%%%%%%%%%%%%%%%%%%%%%%%%%%%%%%%%%%%%%%%

\section{From volume to sections}
	\label{sec:from_volume_to_sections}

\paragraph{Outline of section}
In this section we connect the decay rate of $\vol(D+sA)$ as $s\to 0$ to the growth rate of $h^0(X,\floor{mD}+A)$ as $m\to+\infty$, and use this to complete the proof of \autoref{main}, and to derive \autoref{kappa} and \autoref{scemo}.

\subsection{Recap}
Let $X$ be a Wehler $N$-fold, $N\geq 3$. Thanks to \autoref{thm:volume_of_points} we know that given any $\delta\in [1,\frac{N}{2}]$ we can find a pseudoeffective $\mathbb{R}$-divisor $D$ which satisfies
\begin{align}
	\limsup_{s\downarrow 0} \frac{\log \vol(D+sA)}{\log s}&=\frac{N}{2},\label{supinf2}\\
	\liminf_{s\downarrow 0} \frac{\log \vol(D+sA)}{\log s}&=\delta,\label{supinf2a}
\end{align}
for any ample divisor $A$, which proves \eqref{1} and \eqref{2} in \autoref{main}. To complete the proof of \autoref{main} we need to prove \eqref{3} and \eqref{4}, namely we need to show that if $A$ is sufficiently ample then
\begin{align}
  \liminf_{m \to \infty} \frac{\log h^0(X,\floor{mD} + A)}{\log m} &= \frac{N}{2} \label{supinf3}\\
  \limsup_{m \to \infty} \frac{\log h^0(X,\floor{mD} + A)}{\log m} &= N-\delta. \label{supinf4}
\end{align}
These in fact follow from \eqref{supinf2}, \eqref{supinf2a} together with the very recent \cite[Theorem 1.7, Lemma 2.3]{jiangwang}, which uses sophisticated results from birational geometry due to Birkar \cite{birkar}, and it applies to all Calabi--Yau $3$-folds, as well as higher-dimensional Calabi--Yau manifolds for which the Kawamata--Morrison cone conjecture is known to hold (this includes Wehler $N$-folds thanks to \cite{oguisocantat}). On the other hand, our proof of \eqref{supinf3}, \eqref{supinf4}, which we had found before \cite{jiangwang} was posted, is quite direct and self-contained.

\subsection{Effective nonvanishing on Wehler manifolds}
During the proof of \eqref{supinf3}, \eqref{supinf4}, we will need the following result, which is an analog of \cite[Lemma 5]{lesieutre}:

\begin{proposition}\label{c2x} Let $X$ be a Wehler $N$-fold, $N\geq 3$.
There is a constant $C_N>0$ such that for every nef and big Cartier divisor $G$ on $X$ we have
\begin{equation}\label{gurgh}
\frac{(G^N)}{N!}\leq h^0(X,G)\leq C_N\frac{(G^N)}{N!}.
\end{equation}
\end{proposition}
\noindent An explicit constant $C_N$ can be easily found by tracing through our argument.

\begin{remark}
Observe that \autoref{c2x} implies that every nef and big line bundle on $X$ is effective, a fact which is conjectured by Kawamata to hold for all Calabi--Yau manifolds \cite[Conjecture 2.1]{kawa}. This conjecture was proved in \cite[Theorem 1.6]{caojiang} for a class of Calabi--Yau complete intersections which includes our $X$ as a special case.
\end{remark}

\begin{remark}
Moreover, in \cite[Question 1.8]{caojiang} the authors ask whether all the components $\mathrm{Td}_{2j}$ of the Todd class of a Calabi--Yau manifold are quasi-effective (meaning that they intersect nonnegatively with every nef line bundle). The case of $\mathrm{Td}_2$ is a classical result of Miyaoka and Yau \cite{miyaoka, yau}, and \cite[Theorem 4.6]{caojiang} proves this for $\mathrm{Td}_4$ for complete intersections Calabi--Yau manifolds. Our proof of \autoref{c2x} shows in particular (see \eqref{todd} below) that  the question of Cao-Jiang has an affirmative answer on Wehler manifolds.
\end{remark}

\begin{proof}
We can write
$$G=\sum_{j=1}^{N+1} x_j\omega_j,$$
with $x_j\in\mathbb{N}$.
The intersection form on $X$ has the property that all classes $\omega_j$ square to zero, i.e. $\omega_j^2=0$ in $H^4(X,\mathbb{Z})$, and for the top order intersections we have
\begin{equation}\label{intform}\omega_{i_1}\omega_{i_2}\cdots \omega_{i_N}=\begin{cases}2,\quad \text{if }i_1,\dots,i_N\ \text{distinct},\\
0,\quad\text{otherwise}.
\end{cases}\end{equation}
Since $X$ is Calabi--Yau, Kawamata--Viehweg vanishing gives
$$h^i(X,G)=0,\quad i>0,$$
and Hirzebruch--Riemann--Roch then gives
\begin{equation}\label{hrr}\begin{split}
h^0(X,G) &= \chi(X,G) =\int_X e^{c_1(G)}\mathrm{Td}(X)\\
&=\frac{G^N}{N!}+\sum_{j=1}^N \frac{1}{(N-j)!}(G^{N-j}\cdot \mathrm{Td}_j(X)),
\end{split}\end{equation}
where $\mathrm{Td}_j(X)\in H^{j,j}(X)$ is the degree $2j$ component of the Todd class of $X$. The main task is thus to understand the Todd class. We have the exact sequence
$$0\to TX\to (T(\mathbb{P}^1)^{N+1})\big|_X\to N\to 0,$$
where $N=\mathcal{O}(2,\dots, 2)|_X$ is the normal bundle of $X$ in $(\mathbb{P}^1)^{N+1}$. By the multiplicative property of the Todd class, we see that
\[\begin{split}
\mathrm{Td}(X)&=\mathrm{Td}(T(\mathbb{P}^1)^{N+1}) (\mathrm{Td}(N))^{-1}\\
&=\prod_{j=1}^{N+1}(1+\omega_j) (\mathrm{Td}(N))^{-1},
\end{split}\]
using that $\mathrm{Td}(T\mathbb{P}^1)=1+\omega$ (where $\omega$ is the hyperplane class). It will be useful to denote by $\sigma_p$ the $p$th elementary symmetric function on the variables $\omega_j$ ($0\leq p\leq N$), so that we can write
\[\mathrm{Td}(X)=\left(\sum_{p=0}^N \sigma_p\right)(\mathrm{Td}(N))^{-1}.\]
But recalling now  that any monomial in the variables $\omega_j$ where some variable is raised to a power larger than $1$ vanishes, we can use the multinomial expansion of $\sigma_1^p$ and write for $0\leq p\leq N$,
$$\sigma_p = \frac{\sigma_1^p}{p!},$$
and so
\[\begin{split}
\mathrm{Td}(X)&=\left(\sum_{p=0}^N \frac{\sigma_1^p}{p!}\right)(\mathrm{Td}(N))^{-1}\\
&=e^{\sigma_1}(\mathrm{Td}(N))^{-1}.
\end{split}\]
 The first Chern class of $N$ is equal to
$$c_1(N)=2\sum_{j=1}^{N+1}\omega_j=2\sigma_1,$$
and so from the definition of Todd class we see that
$$\mathrm{Td}(N)=\frac{c_1(N)}{1-e^{-c_1(N)}}=\frac{2\sigma_1}{1-e^{-2\sigma_1}},$$
and so
\[\begin{split}
\mathrm{Td}(X)&=e^{\sigma_1}\frac{1-e^{-2\sigma_1}}{2\sigma_1}\\
&=\frac{\sinh(\sigma_1)}{\sigma_1}\\
&=\sum_{p\geq 0}\frac{\sigma_1^{2p}}{(2p+1)!},
\end{split}\]
from which we see that for all $j\geq 0$
$$\mathrm{Td}_{2j+1}(X)=0,$$
which is a general fact when $c_1(X)=0$ \cite[Remark 1, p.14]{hirz}, and also
\begin{equation}\label{todd}
\mathrm{Td}_{2j}(X)=\frac{\sigma_1^{2j}}{(2j+1)!}=\frac{(2j)!}{(2j+1)!}\sigma_{2j}.
\end{equation}
Inserting this into \eqref{hrr}  gives
\[\begin{split}
h^0(X,G)-\frac{G^N}{N!}&=\sum_{j=1}^{\lfloor N/2\rfloor} \frac{1}{(N-2j)!}(G^{N-2j}\cdot \mathrm{Td}_{2j}(X))\\
&=\sum_{j=1}^{\lfloor N/2\rfloor} \frac{(2j)!}{(2j+1)!(N-2j)!}(G^{N-2j}\cdot\sigma_{2j})\\
&=\sum_{j=1}^{\lfloor N/2\rfloor} \frac{(2j)!}{(2j+1)!(N-2j)!}\left(\left(\sum_{j=1}^{N+1}x_j\omega_j\right)^{N-2j}\cdot\sigma_{2j}\right),
\end{split}\]
and writing $\tau_p$ for the $p$th elementary symmetric function on the variables $x_j$, and recalling again that $\omega_j^2=0$ as well as \eqref{intform}, we can expand
\[\begin{split}
&\left(\left(\sum_{j=1}^{N+1}x_j\omega_j\right)^{N-2j}\cdot\sigma_{2j}\right)\\
&=(N-2j)!\sum_{i_1<\cdots<i_{N-2j}}x_{i_1}\cdots x_{i_{N-2j}}\left(\omega_{i_1}\cdots \omega_{i_{N-2j}}\cdot\sigma_{2j}\right)\\
&=2(N-2j)!\sum_{i_1<\cdots<i_{N-2j}}x_{i_1}\cdots x_{i_{N-2j}}\\
&=2(N-2j)!\tau_{N-2j},
\end{split}\]
and so
\begin{equation}\label{interm}\begin{split}
h^0(X,G)-\frac{G^N}{N!}&=\sum_{j=1}^{\lfloor N/2\rfloor} \frac{2(2j)!}{(2j+1)!}\tau_{N-2j},
\end{split}\end{equation}
and since $x_j\geq 0$, we have $\tau_{N-2j}\geq 0$ and so the first inequality in \eqref{gurgh} follows. As for the second one, observe that
$$\frac{G^N}{N!}=2\tau_N,$$
so the main task is to obtain an upper bound for
$$\frac{\tau_{N-2j}}{\tau_N}=\frac{x_1\cdots x_{N-2j}+\cdots+x_{2j+2}\cdots x_{N+1}}{x_1\cdots x_N+\cdots +x_2\cdots x_{N+1}},$$
where $1\leq j\leq \lfloor N/2\rfloor$.

For this, since $G$ is big (i.e. $\tau_N>0$) at most one of the $x_j$'s can vanish. First assume that they are all strictly positive integers (i.e. $x_j\geq 1$ for all $j$), then each term in the sum in the numerator is bounded above by the whole denominator, and since the numerator contains $\binom{N+1}{2j+1}$ terms we can bound
$$\frac{\tau_{N-2j}}{\tau_N}\leq \binom{N+1}{2j+1}\leq C_N.$$
If on the other hand one of the $x_j$'s vanishes, say $x_{N+1}=0$, then
$$\frac{\tau_{N-2j}}{\tau_N}=\frac{x_1\cdots x_{N-2j}}{x_1\cdots x_N}+\cdots+\frac{x_{2j+1}\cdots x_{N}}{x_1\cdots x_N},$$
and each term in this sum is bounded above by $1$, so on this case we also have
$$\frac{\tau_{N-2j}}{\tau_N}\leq \binom{N}{2j}\leq C_N.$$
Going back to \eqref{interm} we obtain
\[\begin{split}
h^0(X,G)-\frac{G^N}{N!}\leq C_N\sum_{j=1}^{\lfloor N/2\rfloor} \frac{2(2j)!}{(2j+1)!}\tau_N=2C'_N\tau_N=C'_N\frac{G^N}{N!},
\end{split}\]
which proves the other inequality in \eqref{gurgh}.
\end{proof}

\subsection{Completion of the proof of \autoref{main}}

We are now ready to complete the proof of \autoref{main}, by showing \eqref{supinf3} and \eqref{supinf4}.
To start, observe that to define the rounddown $\lfloor mD\rfloor$ we have to choose an actual $\mathbb{R}$-divisor $D$ (not just a numerical equivalence class of $\mathbb{R}$-divisors). However, a simple argument as in \cite[Lemma 2.2]{jiangwang} (originating from \cite[V.1.3, V.2.7 (1)]{nakayama}) shows that the quantities on the LHS of \eqref{supinf3}, \eqref{supinf4} only depend on the numerical equivalence class of $D$. We will make a precise choice presently.

We have now our pseudoeffective $\mathbb{R}$-divisor (numerical class) $D$ such that \eqref{supinf2}, \eqref{supinf2a} hold. We can choose very ample divisors $\{A_j\}$ whose numerical classes form a basis of $N^1(X)$, and then choose an $\mathbb{R}$-divisor representing the class of $D$ which is in the $\mathbb{R}$-linear span of the $A_j$'s. With this choice, given any $m\in\mathbb{N}_{>0}$,
the divisors
$$\varepsilon_m = mD - \floor{mD},$$
are ample (or trivial) and they lie in a compact subset of $N^1(X)_{\mathbb{R}}$. We then choose $A$ sufficiently ample so that $A\geq 2\ve_m$ holds for all $m\geq 1$.

The Cartier divisor $\floor{mD}+A$ is then big for all $m\geq 1$ since we can write it as
\begin{equation}\label{triv}
\floor{mD}+A=\underbrace{mD+\frac{A}{2}}_{\textrm{big}}+\underbrace{\frac{A}{2}-\ve_m}_{\textrm{pseff}}.
\end{equation}
Given $m\geq 1$, thanks to \autoref{prop:characterizing_the_pseudoeffective_cone} we know that there exists a pseudo-automorphism $\phi$ such that the pullback divisor $\phi(\floor{mD}+A)$ is nef. Since pulling back via $\phi$ does not affect the volume or the dimension of spaces of sections, it follows that $\phi(\floor{mD}+A)$ is also big and we have
\[
h^0(X,\floor{mD} + A) = h^0(X,\phi(\floor{mD}+A)).
\]
Thus \[G := \phi(\floor{mD}+A),\] is a big and nef Cartier divisor. To relate the dimension of the space of sections of $G$ to its volume
$$\vol(G)=(G^N),$$
we use \autoref{c2x} which gives
\begin{equation}\label{eq1}\begin{split}
\vol(\floor{mD}+A)&=\vol(\phi(\floor{mD}+A))\leq N! h^0(X,\phi(\floor{mD}+A))\\
&\leq C_N \vol(\phi(\floor{mD}+A))=C_N\vol(\floor{mD}+A).
\end{split}\end{equation}
Recalling that $ h^0(X,\phi(\floor{mD}+A))= h^0(X,\floor{mD}+A)$, this shows that  $h^0(X,\floor{mD}+A)$ is uniformly comparable to $\vol(\floor{mD}+A)$.

\subsubsection{Volume and round-downs}
	\label{sssec:volume_and_round_downs}
Our next goal is to compare $\vol(\floor{mD}+A)$ to $\vol(mD+A)$. One direction is easy, since
$\varepsilon_m = mD - \floor{mD}$ is ample (or trivial), and so
\begin{equation}\label{eq2}
\vol(\floor{mD}+A)\leq \vol(\floor{mD}+A+\ve_m)=\vol(mD+A).
\end{equation}
On the other hand, from \eqref{triv} we see that
\begin{equation}\label{eq3}\begin{split}
\vol(\floor{mD}+A)&\geq \vol\left(mD+\frac{A}{2}\right),
\end{split}
\end{equation}
and combining \eqref{eq1}, \eqref{eq2} and \eqref{eq3} gives
\begin{equation}\label{optimum}
m^N \vol \left( D+\frac{1}{2m} A \right) \leq N!h^0(X,\floor{mD}+A) \leq C_N m^N \vol \left(D+ \frac1m A \right),
\end{equation}
and so
\begin{equation}\label{aq0}
\begin{split} N - \frac{\log \vol \left( D+\frac{1}{2m}A \right)}{\log \left( \frac1m \right)} &\leq \frac{\log h^0(X,\floor{mD}+A)}{\log m}+\frac{\log N!}{\log m} \\
&\leq   \frac{\log C_N}{\log m} + N - \frac{\log \vol \left( D+\frac1m A \right)}{\log \left( \frac1m \right)}.
\end{split}\end{equation}

\subsubsection{The $\liminf$}
	\label{sssec:the_liminf}
We now claim that
\begin{equation}\label{aq1}
\liminf_{m \to \infty} \frac{\log \vol \left( D+\frac1m A \right)}{\log \left( \frac1m \right)}=\delta.
\end{equation}
Indeed, from \eqref{supinf2a} we clearly have that
\begin{equation}\label{clear1}
\liminf_{m \to \infty} \frac{\log \vol \left( D+\frac1m A \right)}{\log \left( \frac1m \right)} \geq \liminf_{s \downarrow 0} \frac{\log \vol \left( D+sA \right)}{\log s}=\delta,
\end{equation}
and we claim that equality holds here. Indeed, recall that using the parameter $t=-\frac{1}{2}\log s,$ the liminf on the RHS is achieved along a fast-growing sequence $t_n\to+\infty$, and furthermore thanks to the discussion in \autoref{sommario} we know that if $t'_n\to+\infty$ is another sequence with $|t'_n-t_n|\leq C$ for some fixed $C$ then the liminf on the RHS is also achieved along $\{t'_n\}$. It then suffices to observe that the sequence
\[t'_n:=\frac{1}{2}\log \lfloor e^{2t_n}\rfloor,\]
satisfies $|t'_n-t_n|\leq C$ for some fixed $C$ and all $n$ large, and so if we take
\[m_n:=e^{2t'_n}=\lfloor e^{2t_n}\rfloor,\]
then we will have
\[\liminf_{n\to+\infty}\frac{\log \vol \left( D+\frac{1}{m_n} A \right)}{\log \left( \frac{1}{m_n} \right)}=\delta,\]
which together with \eqref{clear1} implies \eqref{aq1}.  Arguing in a similar way, we see that
\begin{equation}\label{aq2}
\liminf_{m \to \infty} \frac{\log \vol \left( D+\frac{1}{2m} A \right)}{\log \left( \frac1m \right)}=\delta.
\end{equation}
\subsubsection{The $\limsup$}
	\label{sssec:the_limsup}
Next, we claim that
\begin{equation}\label{aq3}
\limsup_{m \to \infty} \frac{\log \vol \left( D+\frac1m A \right)}{\log \left( \frac1m \right)}=\frac{N}{2}.
\end{equation}
Indeed, from \eqref{supinf2} we have
\begin{equation}\label{clear2}
\limsup_{m \to \infty} \frac{\log \vol \left( D+\frac1m A \right)}{\log \left( \frac1m \right)} \leq \limsup_{s \downarrow 0} \frac{\log \vol \left( D+sA \right)}{\log s}=\frac{N}{2},
\end{equation}
and again we claim that equality holds here. Switching to $t=-\frac{1}{2}\log s,$ the limsup on the RHS is  achieved along some other fast-growing sequence $\ti{t}_n\to+\infty$, and again from \autoref{sommario} we know that this limsup is also achieved along any other sequence $\{\ti{t}'_n\}$ with $|\ti{t}'_n-\ti{t}_n|\leq C$, so if we choose
\begin{equation}\label{zat1}
\ti{t}'_n:=\frac{1}{2}\log \lfloor e^{2\ti{t}_n}\rfloor,
\end{equation}
then $|\ti{t}'_n-\ti{t}_n|\leq C$ for all $n$ large, so if we take
\begin{equation}\label{zat2}
\ti{m}_n:=e^{2\ti{t}'_n}=\lfloor e^{2\ti{t}_n}\rfloor,
\end{equation}
then we will have
\[\limsup_{n\to+\infty}\frac{\log \vol \left( D+\frac{1}{\ti{m}_n} A \right)}{\log \left( \frac{1}{\ti{m}_n} \right)}=\frac{N}{2},\]
which together with \eqref{clear2} implies \eqref{aq3}.  Arguing in a similar way, we see that
\begin{equation}\label{aq4}
\limsup_{m \to \infty} \frac{\log \vol \left( D+\frac{1}{2m} A \right)}{\log \left( \frac1m \right)}=\frac{N}{2}.
\end{equation}
Inserting \eqref{aq1}, \eqref{aq2}, \eqref{aq3} and \eqref{aq4} into \eqref{aq0} finally proves \eqref{supinf3} and \eqref{supinf4}.

\subsection{Numerical dimensions}
In this section we give the proof of \autoref{kappa} and \autoref{scemo}. Here $X,D$ are as in \autoref{coro}, so that \eqref{supinf2}, \eqref{supinf2a}, \eqref{supinf3}, \eqref{supinf4} hold with $\delta=1.$

\begin{proof}[Proof of \autoref{kappa}]
First, from \eqref{cuz1} and \eqref{optimum}, we see that for all $m\geq 1$ we have
\[
N!\,h^0(X,\floor{mD}+A)\geq m^N \vol \left( D+\frac{1}{2m} A \right)\geq C^{-1}m^{\frac{N}{2}},
\]
which implies
\[
\liminf_{m\to+\infty}\frac{h^0(X,\floor{mD}+A)}{m^{\frac{N}{2}}}>0,
\]
and so $\kappa^-_\sigma(D)\geq \left\lfloor\frac{N}{2}\right\rfloor$ and $\kappa^{\mathbb{R},-}_\sigma(D)\geq\frac{N}{2}$. On the other hand, taking the sequence $\ti{s}_n=e^{-2\ti{t}_n}$ such that \eqref{cuz4} holds, define
$\ti{t}'_n$ by \eqref{zat1} and $\ti{m}_n$ by \eqref{zat2}, then  $|\ti{t}'_n-\ti{t}_n|\leq C$ and so
taking $\ti{s}'_n=\frac{1}{\ti{m}_n}$, from  \eqref{cuz4} and \eqref{optimum} we see that
\[h^0(X,\floor{\ti{m}_nD}+A)\leq C_N \ti{m}_n^N \vol \left(D+ \frac{1}{\ti{m}_n} A \right)\leq C\ti{m}_n^{\frac{N}{2}},\]
which implies that for every $\ve\in\mathbb{R}_{>0}$ we have
\[
\liminf_{m\to+\infty}\frac{h^0(X,\floor{mD}+A)}{m^{\frac{N}{2}+\ve}}=0,
\]
and so $\kappa^-_\sigma(D)\leq \left\lfloor\frac{N}{2}\right\rfloor$ and $\kappa^{\mathbb{R},-}_\sigma(D)\leq\frac{N}{2}$, and the first equalities in \eqref{agogn1} and \eqref{agogn2} are proved.

Next, from \eqref{cuz2} and \eqref{optimum}, we see that  for all $m\geq 1$ we have
\[
h^0(X,\floor{mD}+A)\leq C_Nm^N \vol \left( D+\frac{1}{m} A \right)\leq C^{-1}m^{N-1},
\]
which implies
\[
\limsup_{m\to+\infty}\frac{h^0(X,\floor{mD}+A)}{m^{N-1}}<\infty,
\]
and so $\kappa_\sigma(D)\leq \kappa^+_\sigma(D)\leq N-1$ and $\kappa^{\mathbb{R}}_\sigma(D)\leq \kappa^{\mathbb{R},+}_\sigma(D)\leq N-1$. On the other hand, taking the sequence $s_n=e^{-2t_n}$ such that \eqref{cuz3} holds, define
\[t'_n:=\frac{1}{2}\log\left(2\left\lfloor\frac{1}{2}e^{2t_n}\right\rfloor\right),\]
which satisfies $|t'_n-t_n|\leq C$, and
\[m_n:=\left\lfloor\frac{1}{2}e^{2t_n}\right\rfloor,\]
so taking $s'_n=\frac{1}{2m_n}$, from  \eqref{cuz3} and \eqref{optimum} we see that
\[N!\,h^0(X,\floor{m_nD}+A)\geq m_n^N \vol \left(D+ \frac{1}{2m_n} A \right)\geq Cm_n^{N-1-o(1)},\]
which implies that for every $\ve\in\mathbb{R}_{>0}$ we have
\begin{equation}\label{inutil}
\limsup_{m\to+\infty}\frac{h^0(X,\floor{mD}+A)}{m^{N-1-\ve}}=+\infty,
\end{equation}
and so $\kappa^{\mathbb{R},+}_\sigma(D)\geq\kappa^{\mathbb{R}}_\sigma(D)\geq N-1$, and $\kappa_\sigma(D)\geq N-2$.
Observe that \eqref{inutil} also means that if $k\in\mathbb{R}_{\geq 0}$ is such that $$\limsup_{m\to+\infty}\frac{h^0(X,\floor{mD}+A)}{m^{k}}<\infty,$$ then necessarily $k\geq N-1$,
and so $\kappa^+_\sigma(D)\geq N-1$, and the rest of \eqref{agogn1} and \eqref{agogn2} are proved.

Lastly, from \eqref{cuz1} we see that there is $c>0$ such that for all $s>0$ small we have
\[\vol(D+sA)\geq cs^{\frac{N}{2}}\geq c s^{N-\left\lfloor\frac{N}{2}\right\rfloor},\]
from which it follows immediately that $\nu_{\rm vol}(D)\geq\left\lfloor\frac{N}{2}\right\rfloor$ and $\nu^{\mathbb{R}}_{\rm vol}(D)\geq\frac{N}{2}$, while \eqref{cuz4} gives us a sequence $s_n\downarrow 0$ with
\[\vol(D+sA)\leq Cs_n^{\frac{N}{2}},\]
which implies that if $k\in\mathbb{R}_{\geq 0}$ satisfies
\[\vol(D+sA)\geq cs^{N-k},\]
for some $c>0$ and all $s>0$ small, then we must have $k\leq \frac{N}{2}$, and so
$\nu_{\rm vol}(D)\leq\left\lfloor\frac{N}{2}\right\rfloor$ and $\nu^{\mathbb{R}}_{\rm vol}(D)\leq\frac{N}{2}$, and so \eqref{agogn3} is proved.
\end{proof}

To conclude, we turn to the proof of \autoref{scemo}.

\begin{proof}[Proof of \autoref{scemo}]
Suppose \eqref{pirla} holds. Up to enlarging $A$, we may assume without loss that $A$ is sufficiently ample, so that using \eqref{optimum} and \eqref{agogn1} we get
\[
C^{-1}m^{\kappa_\sigma(D)}\leq h^0(X,\floor{mm_0D}+A)\leq C m^N \vol \left(D+ \frac{1}{mm_0} A \right),
\]
for all $m$ sufficiently large, and from \autoref{kappa} we know that $\kappa_\sigma(D)\geq N-2$, so
\begin{equation}\label{ff}
m^N \vol \left(D+ \frac{1}{mm_0} A \right)\geq C^{-1} m^{N-2}.
\end{equation}
Using again \eqref{cuz4}, we find a sequence $\ti{s}_n\to 0$ such that
\begin{equation}\label{f2}
\vol(D+\ti{s}_n A)\leq C \ti{s}_n^{\frac{N}{2}},
\end{equation}
and the same bound holds for any other sequence $\ti{s}'_n\to 0$ such that $\ti{t}_n:=-\frac{1}{2}\log \ti{s}_n, \ti{t}'_n:=-\frac{1}{2}\log \ti{s}'_n$ satisfy $|\ti{t}'_n-\ti{t}_n|\leq C$ for some fixed $C$. Consider then the sequence of intervals $I_n:=[\ti{t}_n-1,\ti{t}_n+1]$, and the sequence of real numbers
\[\hat{t}_m:=\frac{1}{2}\log(mm_0).\]
Since $\hat{t}_m$ grows much slower than linearly, while each $I_n$ has width $2$, it follows that there are sequences $m_j\to+\infty$ and $n_j\to+\infty$ such that $\hat{t}_{m_j}\in I_{n_j}$, so
$$|\hat{t}_{m_j}-\ti{t}_{n_j}|\leq 1.$$
The $s$ parameter that corresponds to $\hat{t}_{m_j}$ is $\frac{1}{m_jm_0}$, so \eqref{f2} applied to this sequence gives
\[\vol\left(D+\frac{1}{m_jm_0} A\right)\leq C m_j^{-\frac{N}{2}},\]
which combined with \eqref{ff} gives
\[
C m_j^{\frac{N}{2}}\geq m_j^N\vol\left(D+\frac{1}{m_jm_0} A\right) \geq C^{-1}m_j^{N-2},
\]
for all $j$ large, which is a contradiction as long as $N\geq 5$.
\end{proof}

%%%%%%%%%%%%%%%%%%%%%%%%%%%%%%%%%%%%%%%%%%%%%%%%%%%%%%%%%%%%%%%%%%%%%%%%%%%%%%%
%%% 				End of Section: From volume to sections
%%%%%%%%%%%%%%%%%%%%%%%%%%%%%%%%%%%%%%%%%%%%%%%%%%%%%%%%%%%%%%%%%%%%%%%%%%%%%%%

\bibliographystyle{sfilip_bibstyle}
\bibliography{volume}

\newcommand{\etalchar}[1]{$^{#1}$}
\providecommand{\bysame}{\leavevmode ---\ }
\providecommand{\andchar}{\&}
\begin{thebibliography}{DGK{\etalchar{+}}21}

\bibitem[BDPP13]{BDPP}
\textsc{Boucksom~S., Demailly~J.-P., P\u{a}un~M.,{ \andchar\hskip
  0.5em}Peternell~T.} --- {``The pseudo-effective cone of a compact
  {K}\"{a}hler manifold and varieties of negative {K}odaira dimension''}. {\em
  J. Algebraic Geom.} {\bfseries 22} no.~2, (2013) 201--248.
  \url{https://doi.org/10.1090/S1056-3911-2012-00574-8}.

\bibitem[Ben00]{Benoist2000_Automorphismes-des-cones-convexes}
\textsc{Benoist~Y.} --- {``Automorphismes des c\^{o}nes convexes''}. {\em
  Invent. Math.} {\bfseries 141} no.~1, (2000) 149--193.
  \url{https://doi.org/10.1007/PL00005789}.

\bibitem[BFJ09]{bfj}
\textsc{Boucksom~S., Favre~C.,{ \andchar\hskip 0.5em}Jonsson~M.} ---
  {``Differentiability of volumes of divisors and a problem of {T}eissier''}.
  {\em J. Algebraic Geom.} {\bfseries 18} no.~2, (2009) 279--308.
  \url{https://doi.org/10.1090/S1056-3911-08-00490-6}.

\bibitem[Bir23]{birkar}
\textsc{Birkar~C.} --- {``Geometry of polarised varieties''}. {\em Publ. Math.
  Inst. Hautes \'{E}tudes Sci.} {\bfseries 137} (2023) 47--105.
  \url{https://doi.org/10.1007/s10240-022-00136-w}.

\bibitem[BKS04]{BKS}
\textsc{Bauer~T., K\"{u}ronya~A.,{ \andchar\hskip 0.5em}Szemberg~T.} ---
  {``Zariski chambers, volumes, and stable base loci''}. {\em J. Reine Angew.
  Math.} {\bfseries 576} (2004) 209--233.
  \url{https://doi.org/10.1515/crll.2004.090}.

\bibitem[Bou02]{boucksom}
\textsc{Boucksom~S.} --- {``On the volume of a line bundle''}. {\em Internat.
  J. Math.} {\bfseries 13} no.~10, (2002) 1043--1063.
  \url{https://doi.org/10.1142/S0129167X02001575}.

\bibitem[CJ20]{caojiang}
\textsc{Cao~Y.{ \andchar\hskip 0.5em}Jiang~C.} --- {``Remarks on {K}awamata's
  effective non-vanishing conjecture for manifolds with trivial first {C}hern
  classes''}. {\em Math. Z.} {\bfseries 296} no.~1-2, (2020) 615--637.
  \url{https://doi.org/10.1007/s00209-019-02455-x}.

\bibitem[CO15]{oguisocantat}
\textsc{Cantat~S.{ \andchar\hskip 0.5em}Oguiso~K.} --- {``Birational
  automorphism groups and the movable cone theorem for {C}alabi-{Y}au manifolds
  of {W}ehler type via universal {C}oxeter groups''}. {\em Amer. J. Math.}
  {\bfseries 137} no.~4, (2015) 1013--1044.
  \url{https://doi.org/10.1353/ajm.2015.0023}.

\bibitem[CP21]{choipark}
\textsc{Choi~S.~R.{ \andchar\hskip 0.5em}Park~J.} --- {``Comparing numerical
  {I}itaka dimensions again''}. \url{https://arxiv.org/abs/2111.00934}.

\bibitem[DGK{\etalchar{+}}21]{DancigerGueritaud_Convex-cocompactness-for-Coxeter-groups2023}
\textsc{Danciger~J., Guéritaud~F., Kassel~F., Lee~G.-S.,{ \andchar\hskip
  0.5em}Marquis~L.} --- {``Convex cocompactness for {C}oxeter groups''}. 2021.
  \url{https://arxiv.org/abs/2102.02757}.

\bibitem[Eck16]{eckl}
\textsc{Eckl~T.} --- {``Numerical analogues of the {K}odaira dimension and the
  abundance conjecture''}. {\em Manuscripta Math.} {\bfseries 150} no.~3-4,
  (2016) 337--356. \url{https://doi.org/10.1007/s00229-015-0815-x}.

\bibitem[ELM{\etalchar{+}}05]{ELMNP}
\textsc{Ein~L., Lazarsfeld~R., Musta\c{t}\v{a}~M., Nakamaye~M.,{ \andchar\hskip
  0.5em}Popa~M.} --- {``Asymptotic invariants of line bundles''}. {\em Pure
  Appl. Math. Q.} {\bfseries 1} no.~2, Special Issue: In memory of Armand
  Borel. Part 1, (2005) 379--403.
  \url{https://doi.org/10.4310/PAMQ.2005.v1.n2.a8}.

\bibitem[FT23]{FT}
\textsc{Filip~S.{ \andchar\hskip 0.5em}Tosatti~V.} --- {``Canonical currents
  and heights for {K}3 surfaces''}. {\em Camb. J. Math.} {\bfseries 11} no.~3,
  (2023) 699--794. \url{https://doi.org/10.4310/cjm.2023.v11.n3.a2}.

\bibitem[Fuj20]{fujino}
\textsc{Fujino~O.} --- {``Corrigendum to ``{O}n subadditivity of the
  logarithmic {K}odaira dimension''''}. {\em J. Math. Soc. Japan} {\bfseries
  72} no.~4, (2020) 1181--1187. \url{https://doi.org/10.2969/jmsj/82568256}.

\bibitem[Hir95]{hirz}
\textsc{Hirzebruch~F.} --- {\em Topological methods in algebraic geometry}.
  Classics in Mathematics. Springer-Verlag, Berlin --- 1995.
  \url{https://doi.org/10.1007/978-3-642-62018-8}.

\bibitem[HS22]{hoffstenger}
\textsc{Hoff~M.{ \andchar\hskip 0.5em}Stenger~I.} --- {``{On the numerical
  dimension of Calabi–Yau 3-folds of Picard number 2}''}. {\em International
  Mathematics Research Notices} (2022) .
  \url{https://doi.org/10.1093/imrn/rnac148}.

\bibitem[HSYn22]{hoffstengeryanez}
\textsc{Hoff~M., Stenger~I.,{ \andchar\hskip 0.5em}Y\'{a}\~{n}ez~J.~I.} ---
  {``Movable cones of complete intersections of multidegree one on products of
  projective spaces''}. \url{https://arxiv.org/abs/2207.11150}.

\bibitem[JW23]{jiangwang}
\textsc{Jiang~C.{ \andchar\hskip 0.5em}Wang~L.} --- {``On numerical dimensions
  of {C}alabi--{Y}au varieties''}. {\em International Mathematics Research
  Notices} (2023) rnad032. \url{https://doi.org/10.1093/imrn/rnad032}.

\bibitem[Kaw97]{Kawamata1997_On-the-cone-of-divisors-of-Calabi-Yau-fiber-spaces}
\textsc{Kawamata~Y.} --- {``On the cone of divisors of {C}alabi-{Y}au fiber
  spaces''}. {\em Internat. J. Math.} {\bfseries 8} no.~5, (1997) 665--687.
  \url{https://doi.org/10.1142/S0129167X97000354}.

\bibitem[Kaw00]{kawa}
\bysame , {``On effective non-vanishing and base-point-freeness''}. {\em Asian
  J. Math.} {\bfseries 4} no.~1, (2000) 173--181.
  \url{https://doi.org/10.4310/AJM.2000.v4.n1.a11}.

\bibitem[KK12]{kk}
\textsc{Kaveh~K.{ \andchar\hskip 0.5em}Khovanskii~A.~G.} ---
  {``Newton-{O}kounkov bodies, semigroups of integral points, graded algebras
  and intersection theory''}. {\em Ann. of Math. (2)} {\bfseries 176} no.~2,
  (2012) 925--978. \url{https://doi.org/10.4007/annals.2012.176.2.5}.

\bibitem[Laz04]{pag1}
\textsc{Lazarsfeld~R.} --- {\em Positivity in algebraic geometry. {I}}, vol.~48
  of {\em Ergebnisse der Mathematik und ihrer Grenzgebiete.} Springer-Verlag,
  Berlin --- 2004. \url{https://doi.org/10.1007/978-3-642-18808-4}.

\bibitem[Leh13]{lehmann}
\textsc{Lehmann~B.} --- {``Comparing numerical dimensions''}. {\em Algebra
  Number Theory} {\bfseries 7} no.~5, (2013) 1065--1100.
  \url{https://doi.org/10.2140/ant.2013.7.1065}.

\bibitem[Les22]{lesieutre}
\textsc{Lesieutre~J.} --- {``Notions of numerical {I}itaka dimension do not
  coincide''}. {\em J. Algebraic Geom.} {\bfseries 31} no.~1, (2022) 113--126.
  \url{https://doi.org/10.1090/jag/763}.

\bibitem[LM09]{lazarsfeldmustata}
\textsc{Lazarsfeld~R.{ \andchar\hskip 0.5em}Musta\c{t}\u{a}~M.} --- {``Convex
  bodies associated to linear series''}. {\em Ann. Sci. \'{E}c. Norm.
  Sup\'{e}r. (4)} {\bfseries 42} no.~5, (2009) 783--835.
  \url{https://doi.org/10.24033/asens.2109}.

\bibitem[LOP18]{LOP}
\textsc{Lazi\'{c}~V., Oguiso~K.,{ \andchar\hskip 0.5em}Peternell~T.} --- {``The
  {M}orrison-{K}awamata cone conjecture and abundance on {R}icci flat
  manifolds''}. in {\em Uniformization, {R}iemann-{H}ilbert correspondence,
  {C}alabi-{Y}au manifolds \& {P}icard-{F}uchs equations} --- vol.~42 of {\em
  Adv. Lect. Math. (ALM)}, pp.~157--185. Int. Press, Somerville, MA --- 2018.
  \url{https://doi.org/10.48550/arXiv.1611.00556}.

\bibitem[McC18]{mccleerey}
\textsc{McCleerey~N.} --- {``Volume of perturbations of pseudoeffective
  classes''}. {\em Pure Appl. Math. Q.} {\bfseries 14} no.~3-4, (2018)
  607--616. \url{https://doi.org/10.4310/PAMQ.2018.v14.n3.a9}.

\bibitem[Miy87]{miyaoka}
\textsc{Miyaoka~Y.} --- {``The {C}hern classes and {K}odaira dimension of a
  minimal variety''}. in {\em Algebraic geometry, {S}endai, 1985} --- vol.~10
  of {\em Adv. Stud. Pure Math.}, pp.~449--476. North-Holland, Amsterdam ---
  1987. \url{https://doi.org/10.2969/aspm/01010449}.

\bibitem[Mor93]{Morrison_Compactifications-of-moduli-spaces-inspired1993}
\textsc{Morrison~D.~R.} --- {``Compactifications of moduli spaces inspired by
  mirror symmetry''}. in {\em Journ\'{e}es de G\'{e}om\'{e}trie Alg\'{e}brique
  d'Orsay (Orsay, 1992)} --- No.~218, pp.~243--271. Ast\'{e}risque --- 1993.
  \url{http://www.numdam.org/item/AST_1993__218__243_0/}.

\bibitem[Nak04]{nakayama}
\textsc{Nakayama~N.} --- {\em Zariski-decomposition and abundance}, vol.~14 of
  {\em MSJ Memoirs}. Mathematical Society of Japan, Tokyo --- 2004.
  \url{https://doi.org/10.2969/msjmemoirs/014010000}.

\bibitem[Rat19]{Ratcliffe2019_Foundations-of-hyperbolic-manifolds}
\textsc{Ratcliffe~J.~G.} --- {\em Foundations of hyperbolic manifolds},
  vol.~149 of {\em Graduate Texts in Mathematics}. Springer, Cham --- 2019.
  \url{https://doi.org/10.1007/978-3-030-31597-9}.

\bibitem[Vin71]{Vinberg1971_Discrete-linear-groups-that-are-generated-by-reflections.}
\textsc{Vinberg~E.~B.} --- {``Discrete linear groups that are generated by
  reflections.''}. {\em Izv. Akad. Nauk SSSR Ser. Mat.} {\bfseries 35} no.~5,
  (1971) 1072--1112. \url{https://doi.org/10.1070/IM1971v005n05ABEH001203}.

\bibitem[Weh88]{wehler}
\textsc{Wehler~J.} --- {``{$K3$}-surfaces with {P}icard number {$2$}''}. {\em
  Arch. Math. (Basel)} {\bfseries 50} no.~1, (1988) 73--82.
  \url{https://doi.org/10.1007/BF01313498}.

\bibitem[WN19]{wittnystrom}
\textsc{Witt~Nystr\"{o}m~D.} --- {``Duality between the pseudoeffective and the
  movable cone on a projective manifold''}. {\em J. Amer. Math. Soc.}
  {\bfseries 32} no.~3, (2019) 675--689.
  \url{https://doi.org/10.1090/jams/922}. With an appendix by S\'{e}bastien
  Boucksom.

\bibitem[Yau77]{yau}
\textsc{Yau~S.~T.} --- {``Calabi's conjecture and some new results in algebraic
  geometry''}. {\em Proc. Nat. Acad. Sci. U.S.A.} {\bfseries 74} no.~5, (1977)
  1798--1799. \url{https://doi.org/10.1073/pnas.74.5.1798}.

\bibitem[Yn22]{yanez}
\textsc{Y\'{a}\~{n}ez~J.~I.} --- {``Birational automorphism groups and the
  movable cone theorem for {C}alabi-{Y}au complete intersections of products of
  projective spaces''}. {\em J. Pure Appl. Algebra} {\bfseries 226} no.~10,
  (2022) Paper No. 107093, 22.
  \url{https://doi.org/10.1016/j.jpaa.2022.107093}.

\end{thebibliography}
\end{document}